\documentclass[a4paper,10pt,reqno]{amsart}

\usepackage[utf8]{inputenc}
\usepackage[T1]{fontenc}
\usepackage{amsthm}
\usepackage{amsmath}
\usepackage{amssymb}
\usepackage[inline]{enumitem}
\usepackage{comment}
\PassOptionsToPackage{hyphens}{url}\usepackage{hyperref}
\usepackage{fancyhdr}
\usepackage{mathrsfs}
\usepackage{stmaryrd}
\usepackage[normalem]{ulem}
\usepackage{xcolor}
\usepackage[skip=0pt]{caption}
\usepackage{tikz-cd}
\usetikzlibrary{arrows}

\theoremstyle{definition}

\newtheorem{theorem}{Theorem}[section]
\newtheorem{lemma}[theorem]{Lemma}
\newtheorem{proposition}[theorem]{Proposition}
\newtheorem{corollary}[theorem]{Corollary}

\theoremstyle{definition}
\newtheorem{definition}[theorem]{Definition}
\newtheorem{example}[theorem]{Example}
\newtheorem{remark}[theorem]{Remark}

\newtheorem{examples}[theorem]{Examples}
\newtheorem{remarks}[theorem]{Remarks}

\definecolor{blue-url}{RGB}{0,0,100}
\definecolor{red-url}{RGB}{100,0,0}
\definecolor{green-url}{RGB}{0,100,0}
\definecolor{light-yellow}{RGB}{255,255,128}
\definecolor{light-blue}{RGB}{193,255,255}
\definecolor{light-red}{RGB}{239,83,80}

\hypersetup{
	pdftitle={Local finiteness conditions},
	pdfauthor={Cossu and Tringali},
	pdfmenubar=false,
	pdffitwindow=true,
	pdfstartview=FitH,
	colorlinks=true,
	linkcolor=blue-url,
	citecolor=green-url,
	urlcolor=red-url
}

\renewcommand{\emptyset}{\varnothing}
\renewcommand{\setminus}{\smallsetminus}
\renewcommand{\,}{\kern 0.1em}

\DeclareMathOperator{\diag}{diag}

\providecommand\llb{\llbracket}
\providecommand\rrb{\rrbracket}
\providecommand\sqeq{\sqsubseteq}
\providecommand\sqneq{\sqsubset}
\providecommand{\RR}{\mathbin{R}}

\newcommand{\evid}[1]{\textsf{#1}}

{\newline\vspace{\abovedisplayskip}\hbox to \textwidth\bgroup\hss$\displaystyle}
{$\hss\egroup\vspace{\belowdisplayskip}}

\makeatletter
\DeclareFontFamily{OMX}{MnSymbolE}{}
\DeclareSymbolFont{MnLargeSymbols}{OMX}{MnSymbolE}{m}{n}
\SetSymbolFont{MnLargeSymbols}{bold}{OMX}{MnSymbolE}{b}{n}
\DeclareFontShape{OMX}{MnSymbolE}{m}{n}{
	<-6>  MnSymbolE5
	<6-7>  MnSymbolE6
	<7-8>  MnSymbolE7
	<8-9>  MnSymbolE8
	<9-10> MnSymbolE9
	<10-12> MnSymbolE10
	<12->   MnSymbolE12
}{}
\DeclareFontShape{OMX}{MnSymbolE}{b}{n}{
	<-6>  MnSymbolE-Bold5
	<6-7>  MnSymbolE-Bold6
	<7-8>  MnSymbolE-Bold7
	<8-9>  MnSymbolE-Bold8
	<9-10> MnSymbolE-Bold9
	<10-12> MnSymbolE-Bold10
	<12->   MnSymbolE-Bold12
}{}

\let\llangle\@undefined
\let\rrangle\@undefined
\DeclareMathDelimiter{\llangle}{\mathopen}%
{MnLargeSymbols}{'164}{MnLargeSymbols}{'164}
\DeclareMathDelimiter{\rrangle}{\mathclose}%
{MnLargeSymbols}{'171}{MnLargeSymbols}{'171}
\makeatother

\begin{document}

\title{Factorization under Local Finiteness Conditions}
\author{Laura Cossu}
\address{(L.C.)~Institute of Mathematics and Scientific Computing, University of Graz | Heinrichstrasse 36/III, 8010 Graz, Austria}
\email{laura.cossu@uni-graz.at}
\urladdr{https://sites.google.com/view/laura-cossu/home-page}
\author{Salvatore Tringali}
\address{(S.T.)~School of Mathematical Sciences,
Hebei Normal University | Shijiazhuang, Hebei province, 050024 China}
\email{salvo.tringali@gmail.com}
\urladdr{http://imsc.uni-graz.at/tringali}
%
%
%
%


\begin{abstract}
\noindent{}It has been recently observed that fundamental aspects of the classical theory of factorization can be greatly generalized by combining the languages of monoids and preorders. This has led to various theorems on the existence of certain factorizations, herein called $\preceq$-factorizations, for the $\preceq$-non-units of a (multiplicatively written) monoid $H$ endowed with a preorder $\preceq$, where an element $u \in \allowbreak H$ is a $\preceq$-unit if $u \preceq 1_H \preceq u$ and a $\preceq$-non-unit otherwise. The ``building blocks'' of these factorizations are the $\preceq$-irreducibles of $H$ (i.e., the $\preceq$-non-units $a \in H$ that cannot be written as a product of two $\preceq$-non-units each of which is strictly $\preceq$-smaller than $a$); and it is interesting to look for sufficient conditions for the $\preceq$-factorizations of a $\preceq$-non-unit to be  bounded in length or finite in number (if measured or counted in a suitable way). This is precisely the kind of questions addressed in the present work, whose main novelty is the study of the interaction between minimal $\preceq$-factorizations (i.e., a refinement of $\preceq$-factorizations used to counter the ``blow-up phenomena'' that are inherent to factorization in non-commutative or non-cancellative monoids) and some finiteness conditions describing the ``local behaviour'' of the pair $(H, \preceq)$. Besides a number of examples and remarks, the paper includes many arithmetic results, a part of which are new already in the basic case where $\preceq$ is the divisibility preorder on $H$ (and hence in the setup of the classical theory). 
\end{abstract}

\maketitle \thispagestyle{empty}

\section{Introduction}
\label{sec:intro}
A number of problems in different areas of mathematics, herein generically named \emph{factorization problems}, revolve around the possibility or impossibility of expressing certain elements of a (multiplicatively written) monoid as a finite product of certain other elements, henceforth referred to as \emph{elementary factors}, that in a sense cannot be ``broken down into smaller pieces''.
One way of formalizing these ideas is to
combine the languages of monoids and preorders, as was recently done in \cite{Co-Tr-21(a), Tr20(c)} as part of a broader program \cite{Fa-Tr18, Tr19a, An-Tr18} aimed to enlarge the boundaries of the classical theory of factorization (see Remark \ref{rem:preorders}\ref{rem:preorders(1)} for a formal definition of what we mean here by this term),
where the elementary factors used all along the ``factorization process'' 
are most usually \evid{atoms} in the sense of P.\,M.~Cohn \cite[p.~587]{Co69} (i.e., non-units that do not factor as a product of two non-units) and the structures taken under consideration (from monoids of modules \cite{Fa02, Wi-Wi09, BaWi13} and monoids of ideals \cite{Ger-Kha22,HK98} to Krull domains and Krull monoids \cite{Ch-Ge97, Ge-Gr09, Ge13}, from 
rings of integer-valued polynomials \cite{Fr-Na-Ri19, Fad-Fri-Win22} to monoid algebras \cite{Cha-Fad-Win22, Fa-Wi22}, from numerical monoids \cite{As-GaSa16, Ge-Schm18, Ge-Schm19, Bl-GaSa-Ge11} to Puiseux monoids \cite{Go18, Ch-Go-Go19}, etc.) are, apart from rare exceptions, cancellative, if not even cancellative and commutative. 

In fact, it is only in very recent years that first significant steps have been made towards a systematic extension of the theory to monoids that need be neither commutative nor cancellative (and hence to rings that need not be domains), although most of the work in this direction has so far been limited to the \emph{unit-cancellative} case \cite{FGKT, Ge-Schw18, GeZh21, Bae-Sme21, Bie-Ger-22} (see Example \ref{exa:irrds-atoms-quarks}\ref{exa:irrds-atoms-quarks(1)} for further details). A major difficulty with factorization in non-unit-cancellative monoids, no matter whether commutative or not, is due to ``blow-up phenomena'' (triggered, e.g., by the presence of proper idempotents and, more generally, non-units of finite order) whose immediate effect is to make a number of (arithmetic) invariants of classical interest essentially meaningless; and analogous phenomena show up in cancellative but non-commutative monoids too (Example \ref{exa:non-atomic-2-generator-1-relator-canc-mon}). The question is extensively discussed in \cite[Sects.~1 and 4]{An-Tr18} and has led to the idea of replacing classical (atomic) factorizations with \emph{minimal} (atomic) factorizations in an effort to counter the issues arising from the departure from cancellativity: Skipping the details for the moment, the main point is again that preorders play a crucial role in the bigger picture.

Our goal in the present paper is to further develop the paradigm of minimal factorizations (Sect.~\ref{sec:minimal-factorizations}) and study fundamental aspects of the arithmetic of monoids that are, in a sense, ``approximately finitely generated on a local scale''. For, we will fully embrace the philosophy of \cite{Co-Tr-21(a), Tr20(c)} and work out most of our results in the language of \emph{premonoids} (that is, monoids endowed with a preorder). 

\subsection*{Plan of the paper.} After reviewing the key definitions of $\preceq$-[non-]unit, $\preceq$-ir\-re\-duc\-i\-ble, and $\preceq$-atom associated with a premonoid $\mathcal H = (H, \preceq)$ and recalling a few notions from the general theory of monoids (Sect. ~\ref{sec:premonoids}), we throw in preordered monoids, weakly (resp., strongly) positive monoids, etc.~(Definition \ref{def:preordered-monoids-et-alia}) and prove a converse (Proposition \ref{prop:FTF-inverso}) to the main theorem of \cite{Tr20(c)} on the existence of a $\preceq$-factorization (that is, a factorization into $\preceq$-irreducibles) for the $\preceq$-non-units of the monoid $H$ under the hypothesis that the preorder $\preceq$ is artinian (Theorem \ref{thm:abstract-factorization}). Next, we bring up [atomic] $\preceq$-fac\-tor\-i\-za\-tions and minimal [atomic] $\preceq$-factorizations and look for sufficient conditions for the [minimal] $\preceq$-factorizations of a $\preceq$-non-unit to be bounded in length or finite in number (if measured or counted in a proper way): This ultimately leads to BF-factorable premonoids, FmF-atomic premonoids, etc.~(Definition \ref{def:factorizations}) and to investigate the interplay between [minimal] [atomic] $\preceq$-factorizations and some finiteness conditions describing the ``local arithmetic'' of $\mathcal H$. Accordingly, we introduce l.f.g., [weakly] l.f.g.u., and loft premonoids (Definition \ref{def:lfgu}) and, among other things, we prove that 
\begin{enumerate*}[label=\textup{(\roman{*})}]
\item every l.f.g.u.~premonoid is weakly l.f.g.u.~(Proposition \ref{pro:f.g.u. is l.f.g.u. is w.l.f.g.u.}),
\item if $\mathcal H$ is a weakly l.f.g.u.~weakly positive monoid, then it is also factorable, i.e., every $\preceq$-non-unit factors as a product of $\preceq$-irreducibles (Corollary \ref{cor:wlfgu weakly positive is factorable}),
\item a loft premonoid is BF- or BmF-factorable if and only if it is, resp., FF- or FmF-factorable (Theorem \ref{thm:loft-BmF-is-FmF}),
\item if $\mathcal H$ is a weakly l.f.g.u.~weakly positive monoid such that every $\preceq$-irreducible is a $\preceq$-atom, then $\mathcal H$ is loft and hence FmF-atomic (Theorem \ref{thm: sufficient for loftness}), and \item every weakly l.f.g.u.~strongly positive monoid is FF-atomic (Corollary \ref{cor:weakly-lfgu-strongly-pos-is-FF-atom}).
\end{enumerate*} 

When $\preceq$ is the divisibility preorder, BF-atomicity and FF-atomicity recover (and generalize) the standard notions of BF-ness and FF-ness \cite[Definitions 1.3.1 and 1.5.1]{GeHK06}; BmF-atomicity and FmF-atomicity recover the notions of BmF-ness and FmF-ness first considered in \cite{An-Tr18}; and, in the com\-mu\-ta\-tive setting, the premonoid $\mathcal H$ is l.f.g.u.~if and only if the monoid $H$ is l.f.g.~after modding out the units (see Remark \ref{rem:lfgu-monoids}\ref{rem:lfgu-monoids(2)} for additional details). We can therefore translate the aforementioned results (on the arithmetic of premonoids) back into the language of the classical theory and hence obtain results that, especially in the case of non-commutative or non-unit-cancellative monoids, are completely new (Theorem \ref{thm:Dedekind-finite wlfgu FmF-atomic} and Corollary \ref{cor:acyclic-is-FmF-atomic}).
In this regard, it is worth remarking that, to a large extent, the classical theory of factorization is all about monoids $H$ that are ``(locally) arithmetically isomorphic'' to a cancellative, commutative, l.f.g.~monoid $K$: The basic idea fits with the abstract philosophy of ``transfer morphisms'' and, in the specific scenario of factorization, boils down to $H$ being \evid{essentially equimorphic} to $K$ in the sense of  \cite[Definition 3.2]{Tr19a}, by which we mean that there is
an atom-preserving monoid homomorphism $f \colon H \to K$ with $K = \allowbreak K^\times f(H)\, K^\times$ such that every atomic factorization (in $K$) of the image $f(x)$ of a non-unit $x \in H$ can be pulled back to an atomic factorization of $x$ 
(see \cite[Remarks 2.17--2.20]{Fa-Tr18} for a critical comparison with analogous definitions from the literature).
Formalities aside, the bottom line is that cancellative, commutative, l.f.g.~monoids have a prominent role in the classical theory: Due to their relative simplicity, they serve as ``canonical models'' for the study of much more complicated objects.
On the other hand, almost nothing is known about the arithmetic of non-commutative or non-unit-cancellative monoids that are not essentially equimorphic to a cancellative, commutative, l.f.g.~monoid, apart from the little that is known in very specific examples: For one thing, every finite monoid is finitely generated (f.g.), but a ma\-jor\-i\-ty of finite monoids are not even atomic and hence, by Remark \ref{rem:lfgu-monoids}\ref{rem:lfgu-monoids(2)}, cannot be essentially e\-qui\-mor\-phic to a cancellative, commutative, l.f.g.~monoid.

We conclude by observing that cancellative f.g.~monoids need not satisfy the ACCP, i.e., the ascending chain condition on principal two-sided ideals (Remark \ref{rem:ACCP} and Example \ref{exa:non-atomic-2-generator-1-relator-canc-mon}): This lies in stark contrast with the case of cancellative, \emph{commutative}, f.g.~monoids \cite[Proposition 2.7.8.4]{GeHK06} and shows, at the end of the day, that Corollary \ref{cor:wlfgu weakly positive is factorable} does not follow from Theorem \ref{thm:abstract-factorization} in any obvious way (Proposition \ref{prop:FTF-inverso} proves that there \emph{is} a way, but the result is more of theoretical interest than of any practical use). The situation is, however, quite different if we focus attention on the class of left (or right) duo monoids (Example \ref{exa:preord-mons}\ref{exa:preord-mons(2)}), which is exactly what we do in the second half of Sect.~\ref{sec:duo-mons} (most notably, we prove in Theorem \ref{thm:lfgu-left-duo-satisfies-ACCP} that every left duo, l.f.g.u.~monoid satisfies the \textup{ACCP}).

Besides a number of examples and remarks that will gently guide the reader through, the paper contains a stack of new ideas, some of which mark a kind of discontinuity with the past (leaving aside the switch from monoids to premonoids in the wake of \cite{Co-Tr-21(a), Tr20(c)}): These include the role of weakly l.f.g.u.~monoids and germs (Definition \ref{def:lfgu}\ref{def:lfgu(1)}) in the ``local analysis'' of arithmetic properties (to the contrary of the classical theory, where the same role is rather played by divisor-closed submonoids); the introduction of weakly positive monoids (as opposite to the stronger notion of preordered monoid) to ``encode'' the arithmetic of (Dedekind-finite) monoids through the language of premonoids (Example \ref{exa:preord-mons}\ref{exa:preord-mons(1)}); and the first-ever applications to factorization theory of a combinatorial result of G.~Higman (see Theorem \ref{thm:higman} and the comments after Lemma \ref{lem:pseudo-commutativity-in-duo-monoids}) that can be regarded as a non-commutative extension of a well-known lemma of L.\,E.~Dickson which is, in turn, an old ac\-quain\-tance of practitioners in the field (see, e.g., Theorem 1.5.3 in \cite{GeHK06}).

\section{Monoids and preorders}
\label{sec:premonoids}

Throughout, $H$ is a multiplicatively written monoid with identity $1_H$ (e.g., the multiplicative monoid of a ring). 
Undefined terminology and notation are either standard or borrowed from \cite[Sect.~2]{Tr20(c)}. In particular, 
we address the reader to Howie's monograph \cite{Ho95} for basic aspects of semigroup theory. We use $H^\times$ for the \evid{group of units} of $H$ and $\langle X \rangle_H$ for the \evid{submonoid} of $H$ \evid{generated} by a set $X \subseteq H$. Accordingly, we say that $H$ is \evid{reduced} if $H^\times$ is trivial (i.e., the only unit is the identity).

Let $\preceq$ a preorder (i.e., a reflexive and transitive binary relation) on (the carrier set of) $H$; in the parlance of \cite[Definitions 3.2 and 3.4]{Tr20(c)}, we refer to the pair $\mathcal H := (H, \preceq)$ as a \evid{premonoid} and say $x \in H$ is \evid{$\preceq$-equivalent} to $y \in H$ if $x \preceq y \preceq x$ (of course, being $\preceq$-equivalent is an equivalence on $H$). The binary relation $\mid_H$ defined by $x \mid_H y$ if and only if $x \in H$ and $y \in HxH$ is in fact a preorder of the utmost importance in the study of factorization: We call $\mid_H$ the \evid{divisibility preorder} (on $H$), let $x$ be a \evid{divisor} of $y$ in $H$ if $x \mid_H y$, and denote by $\llb x \rrb_H$ the \evid{smallest divisor-closed submonoid} of $H$ containing the element $x$.

An el\-e\-ment $u \in H$ is a \evid{$\preceq$-unit} if it is $\preceq$-equivalent to $1_H$; otherwise, $u$ is a \evid{$\preceq$-non-unit}. A $\preceq$-non-unit $a \in \allowbreak H$ is then a \evid{$\preceq$-quark} if there is no $\preceq$-non-unit $b$ with $b \prec a$ (i.e., $b \preceq a$ and $a \not\preceq b$); and is a \evid{$\preceq$-atom of degree $s$} (resp., a \evid{$\preceq$-ir\-re\-duc\-i\-ble} \evid{of degree $s$}), for a certain $s \in \mathbb N_{\ge 2}$, if $a \ne x_1 \cdots x_k$ for every $k \in \llb 2, s \rrb$ and all $\preceq$-non-units $x_1, \ldots, x_k \in H$ (resp., for all $\preceq$-non-units $x_1, \ldots, x_k \in H$ with $x_1 \prec a, \ldots, x_k \prec a$). In par\-tic\-u\-lar, we will simply refer to a $\preceq$-atom of degree $2$ as a \evid{$\preceq$-atom}; to a $\preceq$-ir\-re\-duc\-i\-ble of degree $2$ as a \evid{$\preceq$-ir\-re\-duc\-i\-ble} (note that, occasionally, the term may also be used as an adjective); and to a $\mid_H$-irreducible as an \evid{irreducible} (of $H$). The no\-tions of $\preceq$-[non-]unit, $\preceq$-quark, $\preceq$-atom, and $\preceq$-ir\-red\-u\-ci\-ble were introduced in \cite[Definition 3.6]{Tr20(c)},
while $\preceq$-ir\-re\-duc\-i\-bles of higher degree were first considered in \cite[Definition 3.1]{Co-Tr-21(a)}. We write $\mathcal H^\times$ for the set of $\preceq$-units, $\mathscr A(\mathcal H)$ for the set of $\preceq$-atoms, and $\mathscr I(\mathcal H)$ for the set of $\preceq$-irreducibles. Moreover, given $x \in H$, we denote by $\mathscr I_x(\mathcal H)$ the set of $\preceq_H$-irreducibles $a$ such that $a \mid_H x$; and by $\mathscr A_x(\mathcal H)$ the intersection of $\mathscr I_x(\mathcal H)$ with $\mathscr A(\mathcal H)$.

\begin{remarks}\label{rem:preorders}
\begin{enumerate*}[label=\textup{(\arabic{*})}, mode=unboxed]
\item\label{rem:preorders(1)} Roughly speaking, the classical theory of factorization comes down to the case where $H$ is a \evid{Dedekind-finite} monoid (i.e., the product of any two non-units is a non-unit) and $\preceq$ is the divisibility preorder $\mid_H$ on $H$: A key observation in this regard is that, under the hypothesis of Dedekind-finiteness, a $\mid_H$-unit is an (ordinary) unit of $H$ and hence a $\mid_H$-atom is an (ordinary) atom, cf.~\cite[Remark 3.7]{Tr20(c)}.
\end{enumerate*}

\vskip 0.05cm

\begin{enumerate*}[label=\textup{(\arabic{*})}, mode=unboxed,resume]
\item\label{rem:preorders(2)} Given a premonoid $\mathcal{H}=(H,\preceq)$, we let a \evid{subpremonoid} of $\mathcal H$ be a premonoid $\mathcal K = (K, \preceq_K)$ such that $K$ is a submonoid of $H$ and $\preceq_K$ is the \evid{restriction} of $\preceq$ to $K$ (i.e., the binary relation on $K$ defined by taking $x \preceq_K y$ if and only if $x, y \in K$ and $x \preceq y$): In particular, we denote by $\llb x \rrb_\mathcal{H}$ the subpremonoid of $\mathcal{H}$ whose ``ground monoid'' is the smallest divisor-closed submonoid $\llb x \rrb_H$ of $H$ containing an element $x \in H$. It follows from the definitions that the $\preceq_K$-units of $K$ are exactly the $\preceq$-units of $H$ that lie in $K$, viz., $\mathcal{K}^\times = K \cap \mathcal{H}^\times$: In fact, we have $1_K = 1_H$ (since $K$ is a submonoid of $H$) and hence $u\in K$ is a $\preceq_K$-unit if and only if $1_H\preceq u\preceq 1_H$.
\end{enumerate*}

\vskip 0.05cm

\begin{enumerate*}[label=\textup{(\arabic{*})}, mode=unboxed,resume]
\item\label{rem:preorders(3)} In the notation of item \ref{rem:preorders(2)}, it is worth noting that the restriction to $K$ of the divisibility preorder $\mid_H$ on $H$ is not necessarily the divisibility preorder $\mid_K$ on $K$. This happens, e.g., if $K$ is a divisor-closed submonoid of $H$: If $a,b\in K$ and $a\mid_H b$, then there exist $x, y\in H$ such that $b = xay$; and $K$ being divisor-closed, $x$ and $y$ must lie in $K$ and hence $a \mid_K b$. 
\end{enumerate*}
\end{remarks}

Some simple examples will help us illustrate the notions we have so far introduced: The first of them will also come in handy in Sect.~\ref{sec:duo-mons}, where we focus attention on divisibility.

\begin{examples}
\label{exa:irrds-atoms-quarks}
\begin{enumerate*}[label=\textup{(\arabic{*})}, mode=unboxed]
\item\label{exa:irrds-atoms-quarks(1)}
Understanding the interrelation between the $\preceq$-ir\-re\-duc\-i\-bles, the $\preceq$-atoms, and the $\preceq$-quarks of a monoid $H$, for a given preorder $\preceq$ on $H$, is often pivotal to a deeper comprehension of various phenomena. For instance, it is obvious that $\preceq$-atoms and $\preceq$-quarks are all $\preceq$-ir\-re\-duc\-i\-bles. Yet, a $\preceq$-ir\-red\-u\-ci\-ble need not be either a $\preceq$-atom or a $\preceq$-quark; and neither need a $\preceq$-quark be a $\preceq$-atom. Most notably, this remains true in the fundamental case when $\preceq$ is the divisibility preorder $\mid_H$, see \cite[Remark 3.7(4), Proposition 4.11(iii), and Theorem 4.12]{Tr20(c)}. If, however, the monoid $H$ is \evid{acyclic} in the sense of \cite[Definition 4.2]{Tr20(c)} (namely, $uxv \ne x$ for all $u, v, x \in H$ such that $u$ or $v$ is a non-unit), then we get from \cite[Corollary 4.4]{Tr20(c)} that $\mid_H$-ir\-re\-duc\-i\-bles, $\mid_H$-atoms, $\mid_H$-quarks, and (ordinary) atoms are all the same, an observation that will come in helpful in Corollary \ref{cor:acyclic-is-FmF-atomic}. In the meanwhile, it is worth noting that if $H$ is acyclic or cancellative, then it is also \evid{unit-cancellative} in the sense of \cite[Sect.~2.1, p.~256]{Fa-Tr18} (namely, $xy \ne x \ne yx$ for all $x,y \in H$ with $y \notin H^\times$); the converse need not be true (e.g., see Example \ref{exa:non-atomic-2-generator-1-relator-canc-mon}). Unit-cancellativity was first introduced, in the com\-mu\-ta\-tive setting, in \cite[Sect.~4, p.~72]{RoGaSaGaGa04} (although under a different name) and, independently, in \cite[Sect.~3]{FGKT}. Most literature on factorization in non-cancellative monoids has been so far limited to the unit-cancellative setting. The notions of acyclicity and unit-cancellativity coincide in the commutative setting; otherwise, they are different \cite[Example 4.8]{Tr20(c)}. 
\end{enumerate*}

\vskip 0.05cm

\begin{enumerate*}[label=\textup{(\arabic{*})}, mode=unboxed, resume]
\item\label{exa:irrds-atoms-quarks(2)}
Let $\mathscr P(S)$ be the premonoid obtained by endowing the power set of a set $S$ with the binary operation $\cup_S$ sending a pair of subsets of $S$ to their union and the inclusion order $\subseteq_S$ defined by $X \subseteq_S Y$ if and only if $X \subseteq Y \subseteq S$. It is immediate that the only $\subseteq_S$-unit is the empty set and hence the $\subseteq_S$-ir\-re\-duc\-i\-bles are the one-element subsets of $S$, which, in addition, are all $\subseteq_S$-quarks. On the other hand, the set of $\subseteq_S$-atoms is empty, because $X = X \cup X$ for every set $X$. 
\end{enumerate*}

\vskip 0.05cm

\begin{enumerate*}[label=\textup{(\arabic{*})}, mode=unboxed, resume]
\item\label{exa:irrds-atoms-quarks(3)}
Let $H$ be a monoid and $(K, \preceq)$ a premonoid. Following \cite[Example 3.3(2)]{Tr20(c)}, we define the \evid{pullback} of $\preceq$ \evid{through a function} $\phi \colon H \to K$ as the binary relation $\preceq_\phi$ on $H$ such that $x \preceq_\phi y$ if and only if $\phi(x) \preceq \allowbreak \phi(y)$. It is routine to check that $\preceq_\phi$ is a preorder on $H$, and we seek a ``sensible characterization'' of the $\preceq_\phi$-units, $\preceq_\phi$-ir\-re\-duc\-i\-bles, $\preceq_\phi$-atoms, and $\preceq_\phi$-quarks in terms of the $\preceq$-units, $\preceq$-ir\-re\-duc\-i\-bles, $\preceq$-atoms, and $\preceq$-quarks of $K$, resp. In general, this is unattainable. Assume, however, that $\phi$ is a monoid isomorphism from $H$ to $K$: An element $u \in H$ is then a $\preceq_\phi$-unit if and only if $\phi(u)$ is a $\preceq$-unit; and an element $a \in H$ is a $\preceq_\phi$-ir\-re\-duc\-i\-ble, a $\preceq_\phi$-atom, or a $\preceq_\phi$-quark if and only if $\phi(a)$ is, resp., a $\preceq$-ir\-re\-duc\-i\-ble, a $\preceq$-atom, or a $\preceq$-quark of $K$ (we leave the details to the reader).
\end{enumerate*}
\end{examples}

The reader will have noted by now that, in the definition of a premonoid $(H, \preceq)$, no compatibility between the operation of the monoid $H$ and the preorder $\preceq$ is assumed. This is intentional, because in general no compatibility is guaranteed, for instance, in the fundamental case of the divisibility preorder \cite[Example 3.5(1)]{Tr20(c)}. However, an interplay between $H$ and $\preceq$ (as weak as it may be) is necessary in certain applications, which leads us straight to:

\begin{definition}\label{def:preordered-monoids-et-alia}
A premonoid $\mathcal H = (H, \preceq)$ is a \evid{preordered} (resp., \evid{strongly preordered}) \evid{monoid} if $x \preceq y$ implies $uxv \preceq uyv$ (resp., $\mathcal H$ is preordered and $x \prec y$ implies $uxv \prec uyv$) for all $u, v \in H$; a \evid{positive} (resp., \evid{strongly positive}) \evid{monoid} if $\mathcal H$ is a preordered (resp., strongly preordered) monoid with the further property that $1_H \preceq H$; and a \evid{weakly positive monoid} if $\mathcal H^\times x \,\mathcal H^\times \preceq x \preceq H x H$ for every $x \in H$, where we write $x$ in place of $\{x\}$ and $A \preceq B$, for some $A, B \subseteq H$, means that $a \preceq b$ for all $a \in A$ and $b \in B$.
\end{definition}

In the next remarks, we list some properties of preordered or weakly (resp., strongly) positive monoids that will come in handy later (e.g., in the proofs of Theorem \ref{thm:fgu weakly positive} and Corollary \ref{cor:weakly-lfgu-strongly-pos-is-FF-atom}), while the sub\-se\-quent examples will further clarify why, in a sense, the definition of a premonoid \emph{has to} be so general.

\begin{remarks}
\label{rem:premonoids}
\begin{enumerate*}[label=\textup{(\arabic{*})}, mode=unboxed]
\item\label{rem:premonoids(1)}
Of course, every strongly preordered (resp., strongly positive) monoid is also a preordered (resp., positive) monoid. In general, neither of these implications can be reversed. However, if $\mathcal H = (H, \preceq)$ is a preordered monoid with the additional property that $H$ is cancellative and $\preceq$ is an order (i.e., an antisymmetric preorder), then $\mathcal H$ is in fact a strongly preordered monoid: If $x \prec y$ and $u, v \in \allowbreak H$, then $uxv \preceq uyv$ (since $\mathcal H$ is a preordered monoid); and if the latter inequality is \emph{not} strict, then $uxv = \allowbreak uyv$ (since $\preceq$ is an order) and hence $x = y$ (since $H$ is cancellative), which is a contradiction.
\end{enumerate*}

\vskip 0.05cm

\begin{enumerate*}[label=\textup{(\arabic{*})}, mode=unboxed, resume]
\item\label{rem:premonoids(2)} Given a preordered monoid $\mathcal{H}=(H, \preceq)$, it is found (by induction on $n$) that, if $x_1 \preceq y_1, \ldots, x_n \preceq y_n$, then $x_1 \cdots x_n \preceq y_1 \cdots y_n$. It follows that, if $1_H \preceq y_i$ for each $i \in \allowbreak \llb 1, n \rrb$ and $\sigma$ is a strictly increasing func\-tion $\llb 1, k \rrb \to \llb 1, n \rrb$ (with $k \in \mathbb N$), then $
1_H \preceq \prod_{i=1}^k y_{\sigma(i)} \preceq y_1 \cdots y_n$,
with strict inequality on the left if and only if $1_H \prec y_{\sigma(i)}$ for some $i \in \llb 1, k \rrb$, and strict inequality on the right if $\mathcal H$ is a strongly positive monoid, $y_i$ is a $\preceq$-non-unit for all $i \in \llb 1, n \rrb$, and $k < n$ (we leave the details to the reader). In particular, this yields that, if $\mathcal{H}$ is strongly positive, then  every $\preceq$-irreducible is a $\preceq$-atom and hence $\mathscr I(\mathcal H) = \mathscr A(\mathcal H)$.
\end{enumerate*}
\vskip 0.05cm
\begin{enumerate*}[label=\textup{(\arabic{*})}, resume, mode=unboxed]
\item\label{rem:premonoids(3)} If $\mathcal H = (H, \preceq)$ is a weakly positive monoid, then $1_H \preceq H 1_H H = H$ and the product of any two $\preceq$-units is itself a $\preceq$-unit (i.e., $\mathcal H^\times$ is a submonoid of $H$). Also, the set of $\preceq$-non-units is a two-sided ideal of $H$ (i.e., a subset $\mathfrak i$ of $H$ with the property that $H \mathfrak i\, H \subseteq \mathfrak i$): In fact, if $u$, $v$, and $x$ are elements of $H$ with $x \notin \mathcal H^\times$ and $uxv$ is a $\preceq$-unit, then $uxv \preceq x (uxv) \preceq \allowbreak x \preceq \allowbreak uxv$ and hence $x \in \mathcal H^\times$ (absurd). 
\end{enumerate*}
\vskip 0.05cm
\begin{enumerate*}[label=\textup{(\arabic{*})}, resume, mode=unboxed]
\item\label{rem:premonoids(4)} A subpremonoid $\mathcal K = (K, \preceq_K)$ of a weakly positive monoid $\mathcal H = (H, \preceq)$ is itself a weakly positive monoid. In fact, we have from Remark \ref{rem:preorders}\ref{rem:preorders(2)} that every $\preceq_K$-unit is a $\preceq$-unit. So, $\mathcal H$ being a weakly positive monoid implies that $\mathcal K^\times x \,\mathcal K^\times \preceq x \preceq KxK$ for every $x \in K$, which shows in turn that $\mathcal K^\times x \,\mathcal K^\times \preceq_K x \preceq_K KxK$ (and finishes the proof) because $\mathcal K^\times x \,\mathcal K^\times \subseteq KxK \subseteq K$.\\

\indent{}It is even easier to prove that a subpremonoid of a preordered (resp., strongly preordered) monoid is still a preordered (resp., strongly preordered) monoid. In particular, a subpremonoid of a positive (resp., strongly positive) monoid is positive (resp., strongly positive).
\end{enumerate*}
\vskip 0.05cm
\begin{enumerate*}[label=\textup{(\arabic{*})}, resume, mode=unboxed]
\item\label{rem:premonoids(5)} If $\mathcal H = (H, \preceq)$ is a positive monoid and $u \in H$ is an (ordinary) unit, then $1_H \preceq u^{-1}$ and hence $1_H \preceq \allowbreak u = \allowbreak 1_H u \preceq \allowbreak u^{-1} u = 1_H$, which ultimately proves that $H^\times \subseteq \mathcal{H}^\times$. Moreover, $\mathcal H$ is a weakly positive monoid. For, we already know from item \ref{rem:premonoids(2)} that $x \preceq HxH$ for every $x \in H$ (recall that, in a positive monoid, $1_H \preceq H$), and it remains to see that $\mathcal H^\times x \, \mathcal H^\times \preceq x$. For, let $u,v\in \mathcal{H}^\times$. Since $u\preceq 1_H$, then $ux \preceq \allowbreak x$ and $uxv \preceq xv$. But $v \preceq 1_H$ implies $xv \preceq x$, and hence $uxv \preceq x$ (as wished).
\end{enumerate*}
\end{remarks}

\begin{examples}
\label{exa:preord-mons}
\begin{enumerate*}[label=\textup{(\arabic{*})}, mode=unboxed]
\item\label{exa:preord-mons(1)}
The divisibility preorder $\mid_H$ on a monoid $H$ has the obvious property that $x \mid_H HxH$ for all $x \in H$ and hence $1_H \mid_H H$. Thus it follows by Remark \ref{rem:preorders}\ref{rem:preorders(1)} that if $H$ is Dedekind-finite, then $(H,\mid_H)$ is a weakly positive monoid. The converse, however, need not be true.\\

\indent{}For, let $M$ be a non-Dedekind-finite monoid so that we can pick $x,y\in M$ with $xy=1_M\ne yx$, and define $H$ as the submonoid of $M$ generated by $x$ and $y$. It is evident that $H$ is not Dedekind-finite, and we get from \cite[Remark 4.9(2)]{Tr20(c)} that $H=HzH$ for every $z\in H$. In consequence, $(H, \mid_H)$ is a positive monoid (with the additional property that every element is a $\mid_H$-unit), which suffices to conclude because every positive monoid is weakly positive by Remark \ref{rem:premonoids}\ref{rem:premonoids(5)}.
\end{enumerate*}

\vskip 0.05cm

\begin{enumerate*}[label=\textup{(\arabic{*})}, mode=unboxed, resume]
\item\label{exa:preord-mons(2)}
In analogy with the case of rings (see, e.g., \cite{Mark04} and references therein), a monoid $H$ is called \evid{left} (resp., \evid{right}) \evid{duo} if $aH \subseteq Ha$ (resp., $Ha \subseteq aH$) for all $a \in H$; and is \evid{duo} if it is left and right duo. Duo monoids (also known as \emph{normal} or \emph{normalizing} monoids) and, more generally, ``one-sided duo'' monoids have much in common with commutative monoids (every commutative monoid is obviously duo) and we will come back to them in Sect.~\ref{sec:duo-mons}. Here we just note, in complement to item \ref{exa:preord-mons(1)}, that the Dedekind-finiteness of the monoid $H$ is not enough for the pair $(H, \mid_H)$ to be a pre\-ordered and hence positive monoid (see, e.g., Example 3.5(1) in \cite{Tr20(c)}): A sufficient condition for this to happen is that $H$ is left or right duo, for then $u \mid_H x$ and $v \mid_H y$ imply $xy \in HuH \cdot HvH \subseteq HuvH$ and hence $uv \mid_H xy$. If, in addition, $H$ is commutative and unit-cancellative, then $(H, \mid_H)$ is strongly preordered and hence strongly positive (we leave the details to the reader).
\end{enumerate*}

\vskip 0.05cm

\begin{enumerate*}[label=\textup{(\arabic{*})}, mode=unboxed, resume]
\item\label{exa:preord-mons(3)}
Let $G$ be a \evid{totally orderable} group (written multiplicatively), by which we mean that there exists a \emph{total} order $\preceq$ on $G$ such that the pair $\mathcal G = (G, \preceq)$ is a preordered monoid; for an extensive list of totally orderable groups, see \cite{Lin-Ret-Rol09} and references therein. The set $\{x \in G \colon 1_G \preceq x\}$, endowed with the restriction of the order $\preceq$, is then a subpremonoid of $\mathcal G$, herein denoted by $\mathrm{Con}(\mathcal G)$ and called the \evid{non-negative cone} of $\mathcal G$. It is routine to check that $\mathrm{Con}(\mathcal G)$ is, in fact, a strongly positive monoid (note that every submonoid of a group is cancellative and hence Remark \ref{rem:premonoids}\ref{rem:premonoids(1)} applies).
\end{enumerate*}
\end{examples}

Following \cite[Definitions 3.8 and 3.11]{Tr20(c)}, we say that a preorder $\preceq$ on a monoid $H$ is \evid{artinian}, or $H$ is a \evid{$\preceq$-artinian} monoid, if there is no infinite sequence $x_1, x_2, \ldots$ of elements of $H$ with $x_{n+1} \prec x_n$ for each $n \in \mathbb N^+$; and we take the \evid{$\preceq$-height} of an element $x \in H$ to be the supremum of the set of all $n \in \allowbreak \mathbb N^+$ for which there are $\preceq$-non-units $x_1, \ldots, \allowbreak x_n \in \allowbreak H$ with $x_1 = x$ and $x_{k+1} \prec x_k$ for each $k \in \llb 1, n-1 \rrb$ (with the understanding that $\sup \emptyset := 0$). We call $H$ a \evid{strongly $\preceq$-artinian} monoid, or $\preceq$ a \evid{strongly artinian} preorder on $H$, if every element of $H$ has a finite $\preceq$-height.

\begin{remark}\label{rem:ACCP} 
A monoid $H$ is said to satisfy the \evid{ascending chain condition} (ACC) \evid{on principal two-sided ideals} (ACCP) if there is no (infinite) sequence $x_1, x_2, \ldots$ of elements of $H$ such that $Hx_i H \subsetneq Hx_{i+1} H$ for each $i \in \mathbb N^+$. As already observed in \cite[Remark 3.9.4]{Tr20(c)}, $H$ satisfies the ACCP if and only if the di\-vis\-i\-bil\-i\-ty preorder $\mid_H$ is artinian. Thus, we get from Remark \ref{rem:preorders}\ref{rem:preorders(3)} that $H$ satisfies the ACCP if and only if so does every divisor-closed submonoid of $H$, if and only if so does $\llb x \rrb_H$ for every $x \in H$. 

For, note in particular that, if $x_1, x_2, \ldots$ is a sequence such that $x_{i+1} \mid_H x_i$ for each $i \in \mathbb N^+$, then $x_1, x_2, \ldots$ lie all in $\llb x_1 \rrb_H$ and, in fact, $x_{i+1} \mid_{\llb x_1 \rrb_H} x_i$ for every $i \in \mathbb N^+$.
So, if $\llb x \rrb_H$ satisfies the ACCP for all $x \in H$, then $x_i \mid_{\llb x_1 \rrb_H} x_{i+1}$ and hence $x_i \mid_H x_{i+1}$ for all large $i \in \mathbb N^+$, which shows $H$ is $\mid_H$-artinian.
\end{remark}

The significance of these definitions is related to the next result and its refinements \cite[Theorems 3.4 and 3.5]{Co-Tr-21(a)}, which can be effectively applied to a variety of situations where the goal is merely to prove the \emph{existence} of certain factorizations, decompositions, etc.

\begin{theorem}\label{thm:abstract-factorization}
If $H$ is a $\preceq$-artinian monoid and $s$ is an integer $\ge 2$, then every $\preceq$-non-unit $x$ fac\-tors as a product of $s^{\mathrm{ht}(x) - 1}$ or fewer $\preceq$-ir\-re\-duc\-i\-bles of degree $s$, where $\mathrm{ht}(x)$ is the $\preceq$-height of $x$.
\end{theorem}

To date, applications of Theorem \ref{thm:abstract-factorization} include a generalization to unit-cancellative monoids \cite[Corollary 4.6]{Tr20(c)} of a classical theorem of Cohn \cite[Proposition 0.9.3]{Co06} on \emph{atomic factorizations} (i.e., factorizations into atoms) in cancellative monoids; a non-commutative generalization \cite[Corollary 4.1]{Tr20(c)} of a factorization theorem of D.\,D.~Anderson and S.~Valdes-Leon \cite[Theorem 3.2]{AnVL96} for commutative rings; a refinement \cite[Proposition 4.11 and Theorem 4.12]{Tr20(c)} of a characterization theorem of A.\,A.~Antoniou and Tringali \cite[Theorem 3.9]{An-Tr18} on atomic factorizations in various ``monoids of sets'' naturally arising from arithmetic combinatorics (see Example \ref{exa:lfgu-premonoids}\ref{exa:lfgu-premonoids(2)} and references therein); and a quantitative strengthening \cite[Theorem 5.19]{Co-Tr-21(a)} of the classical theorems of J.\,A.~Erdos \cite{Er68}, R.\,J.\,H.~Dawlings \cite{Da81}, T.\,J.~Laffey \cite[Theorem 1]{Laf83}, and J.~Fountain \cite[Theorem 4.6]{Fo91} on \emph{idempotent fac\-tor\-i\-za\-tions} (i.e., factorizations into idempotents) in the multiplicative monoid of the ring of $n$-by-$n$ matrices over a skew field or a commutative DVD. 

By and large, the object of the present paper is to discuss further applications of the same ideas. We begin with a result that can be viewed as a sort of converse to Theorem \ref{thm:abstract-factorization}.

\begin{proposition}\label{prop:FTF-inverso}
Let $A$ and $S$ be subsets of a monoid $H$ with $1_H \notin A \cup S$. The following are equivalent:
\begin{enumerate}[label=\textup{(\alph{*})}]
    \item\label{prop:FTF-inverso(a)}  Every element of $S$ factors as a non-empty product of elements of $A$.
    \item\label{prop:FTF-inverso(b)} There exists a strongly artinian preorder $\preceq$ on $H$ such that every $x \in S$ is a $\preceq$-non-unit and an element $a \in H$ is $\preceq$-ir\-re\-duc\-i\-ble if and only if it is a $\preceq$-quark, if and only if $a \in A$.
    \item\label{prop:FTF-inverso(c)} There exists an artinian preorder $\preceq$ on $H$ such that every $x \in S$ is a $\preceq$-non-unit and every $\preceq$-ir\-re\-duc\-i\-ble is an element of $A$.
\end{enumerate}
\end{proposition}

\begin{proof}
We will focus attention on proving that \ref{prop:FTF-inverso(a)} $\Rightarrow$ \ref{prop:FTF-inverso(b)}, because the implication \ref{prop:FTF-inverso(b)} $\Rightarrow$ \ref{prop:FTF-inverso(c)} is obvious and \ref{prop:FTF-inverso(c)} $\Rightarrow$ \ref{prop:FTF-inverso(a)} is a trivial consequence of Theorem \ref{thm:abstract-factorization}.

\vskip 0.05cm

\ref{prop:FTF-inverso(a)} $\Rightarrow$ \ref{prop:FTF-inverso(b)}:  
In the language of Example \ref{exa:irrds-atoms-quarks}\ref{exa:irrds-atoms-quarks(3)}, let $\preceq$ be the pullback of the usual order $\leq$ on $\mathbb N$ through the function $\phi \colon H \to \mathbb N$ that maps a non-identity element $x \in \langle A \rangle_H$ to the smallest integer $n \ge 1$ such that $x = a_1 \cdots a_n$ for some $a_1, \ldots, a_n \in A$ and any other element of $H$ to $0$. 
Since $(\mathbb N, +)$ is a strongly $\leq$-artinian monoid (by the well-ordering principle), $\preceq$ is a strongly artinian preorder on $H$. Moreover, $u \in H$ is a $\preceq$-unit if and only if $\phi(u) = 0$. So, every $x \in S$ is a $\preceq$-non-unit (by the hypothesis that $1_H \notin \allowbreak S$ and the assumption that $x$ factors as a non-empty product of elements of $A$), and every $a \in A$ is a $\preceq$-quark (because $\phi(a) = 1$ and hence $b \prec a$ only if $b$ is a $\preceq$-unit). Recalling that a $\preceq$-quark is ob\-vi\-ous\-ly a $\preceq$-ir\-re\-duc\-i\-ble, it remains to see that all $\preceq$-ir\-re\-duc\-i\-bles are in $A$.

For, let $x \in H$ be $\preceq$-ir\-re\-duc\-i\-ble. In particular, this means that $x$ is a $\preceq$-non-unit, with the result that $n := \phi(x) \in \mathbb N^+$ and $x = a_1 \cdots a_n$ for some $a_1, \ldots, a_n \in A$. Suppose for a contradiction that $x \notin A$. Then $n \ge 2$ and hence $x = yz$, where $y := a_1 \cdots a_{n-1}$ and $z := a_n$ are $\preceq$-non-units with $1 \le \phi(y) \le \allowbreak n-1 < n$ and $1 = \allowbreak \phi(z) < n$, namely, $y \prec x$ and $z \prec x$ (if $\phi(y) = 0$, then $y = 1_H$ and $x = a_n \in A$, which is absurd). This, however, means that $x$ is not $\preceq$-ir\-re\-duc\-i\-ble, which is impossible and finishes the proof.
\end{proof}

Proposition \ref{prop:FTF-inverso} shows that Theorem \ref{thm:abstract-factorization} is, in a certain sense, ``best possible'': In a monoid, proving that every element of a given set $S$ factors through the elements of a prescribed set $A$ of elementary factors is equivalent to the artinianity of a suitable preorder. The result is of more theoretical than
practical interest, but the reader will hopefully agree that it adds to the ``roundness and soundness'' of the ideas put forth in the present work and its predecessors \cite{Tr20(c), Co-Tr-21(a)}.

With that said, there are in fact many possible ways to improve on Theorem \ref{thm:abstract-factorization} when attention is focused on a specific class of premonoids. Similarly as in the classical theory, one may want to check, e.g., if the factorizations (however defined) of a fixed element are all ``bounded in length'', or if there are only finitely many of them that are ``essentially different'', or if the same conditions hold true for some ``restricted class'' of factorizations. Formalizing these ideas will keep us busy in the next section. 

\section{Factorizations and minimal factorizations}
\label{sec:minimal-factorizations}
Given a set $X$, we denote by $\mathscr F(X)$ the free monoid on $X$; use the symbols $\ast_X$ and $\varepsilon_X$, resp., for the operation and the identity of $\mathscr F(X)$; and refer to an element of $\mathscr F(X)$ as an \evid{$X$-word}, or simply as a \evid{word} if no confusion can arise. 
We recall that $\mathscr F(X)$ consists, as a set, of all finite tuples of elements of $X$; and $\mathfrak u \ast_X \mathfrak v$ is the \evid{con\-cat\-e\-na\-tion} of two such tuples $\mathfrak u$ and $\mathfrak v$. Accordingly, the identity of $\mathscr F(X)$ is the empty tuple, herein called the \evid{empty $X$-word} (or simply the \evid{empty word} if $X$ is clear from the context). 

We take the (\evid{word}) \evid{length} of an $X$-word $\mathfrak u$, denoted by $\|\mathfrak u\|_X$, to be the unique non-negative integer $h$ such that $\mathfrak u \in X^{\times h}$ (so the empty word is the only $X$-word whose length is zero). Note that, if $\mathfrak u$ is an $X$-word of positive length $h$, then $\mathfrak u = u_1 \ast_X \cdots \ast_X u_h$ for some uniquely de\-ter\-mined $u_1,\, \ldots,\, \allowbreak u_h \in \allowbreak X$, named the \evid{letters} of the word. Given $i \in \llb 1, h \rrb$, we then denote by $\mathfrak u[i]$ the $i^\text{th}$ letter of $\mathfrak u$; whence we have $\mathfrak u = \mathfrak u[1] \ast \cdots \ast \mathfrak u[h]$ provided that $\mathfrak u$ is not the empty word. 

When there is no serious risk of ambiguity, we will usually drop the ``$X$'' from the above notation and write $\mathfrak u^{\ast k}$ for the $k^\mathrm{th}$ power of an $X$-word $\mathfrak u$ (so that $\mathfrak u^{\ast 0} := \varepsilon$ and $\mathfrak u^{\ast (k+1)} := \mathfrak u^{\ast k} \ast \mathfrak u$). 

Building on these premises, we aim to extend an approach first envisioned in \cite[Section~4]{An-Tr18}
to counter the ``blow-up phenomena'' already mentioned in the introduction. The starting point is the following:

\begin{definition}
\label{def:shuffling-preorder}
Given a premonoid $\mathcal H = (H, \preceq)$, we denote by $\sqeq_\mathcal{H}$  
the binary re\-la\-tion on the free monoid $\mathscr F(H)$ defined by $\mathfrak a \sqeq_\mathcal{H} \mathfrak b$, for some $H$-words $\mathfrak a$ and $\mathfrak b$, if and only if there is an injective function $\sigma \colon \llb 1, \|\mathfrak a\|_H \rrb \to \llb 1, \|\mathfrak b\|_H \rrb$ such that $\mathfrak a[i] \preceq \mathfrak b[\sigma(i)] \preceq \mathfrak a[i] $ for every $i \in \llb 1, \|\mathfrak a\|_H \rrb$.
\end{definition}

Since the composition of two injective functions is injective, it is immediate that the relation $\sqeq_\mathcal{H}$ in Definition \ref{def:factorizations} is in fact a preorder on the free monoid $\mathscr F(H)$; more precisely, $\sqeq_\mathcal{H}$ is an artinian preorder, because $\mathfrak a \sqneq_\mathcal{H} \mathfrak b$ implies $\|\mathfrak a\|_H < \|\mathfrak b\|_H$ (it is easily seen that, if $\mathfrak a \sqeq_\mathcal{H} \mathfrak b$ and $\|\mathfrak a\|_H = \|\mathfrak b\|_H$, then $\mathfrak b \sqeq_\mathcal{H} \mathfrak a$). This makes it possible to talk of $\sqeq_\mathcal{H}$-minimality and $\sqeq_\mathcal{H}$-equivalence in $\mathscr F(H)$, so leading to:

\begin{definition}\label{def:factorizations}
\begin{enumerate*}[label=\textup{(\arabic{*})}, mode=unboxed]
\item\label{def:factorizations(1)} A \evid{$\preceq$-factorization} of an element $x \in H$ is an $\mathscr I(\mathcal H)$-word $\mathfrak a \in \pi_H^{-1}(x)$ and we set $
\mathcal{Z}_{\mathcal H}(x) := \allowbreak \pi_H^{-1}(x) \cap \allowbreak \mathscr F(\mathscr I(\mathcal H))$, where $\pi_H$ is the \evid{factorization homomorphism} of $H$, i.e., the unique extension of the identity map on $H$ to a monoid homomorphism $\mathscr F(H) \to H$.
\end{enumerate*}

\vskip 0.05cm

\begin{enumerate*}[label=\textup{(\arabic{*})}, resume, mode=unboxed]
\item\label{def:factorizations(2)}
A \evid{minimal $\preceq$-factorization} of $x$ is then a $\sqeq_\mathcal{H}$-minimal word in $\mathcal{Z}_{\mathcal H}(x)$, namely, an $\mathscr I(\mathcal H)$-word $\mathfrak a \in \allowbreak \pi_H^{-1}(x)$ with the additional property that there is no $\mathscr I(\mathcal H)$-word $\mathfrak b \in \pi_H^{-1}(x)$ such that $\mathfrak b  \sqneq_\mathcal{H} \mathfrak a$. We denote the set of minimal $\preceq$-factorizations of $x$ by $
\mathcal{Z}_{\mathcal H}^{\sf m}(x)$; and we refer to
\[
\mathsf{L}_{\mathcal H}(x) := \bigl\{ \|\mathfrak{a}\|_H: \mathfrak{a} \in \mathcal{Z}_{\mathcal H}(x) \bigr\}
\quad\text{and}\quad
\mathsf{L}_{\mathcal H}^{\sf m}(x) := \bigl\{ \|\mathfrak{a}\|_H: \mathfrak{a} \in \mathcal{Z}_{\mathcal H}^{\sf m}(x) \bigr\},
\]
resp., as the \evid{set of lengths} and the \evid{set of minimal lengths} of $x$ (relative to the premonoid $\mathcal H$).
\end{enumerate*}

\vskip 0.05cm

\begin{enumerate*}[label=\textup{(\arabic{*})}, resume, mode=unboxed]
\item\label{def:factorizations(3)} 
Likewise, an \evid{atomic $\preceq$-factorization} of $x$ is an $\mathscr A(\mathcal H)$-word $\mathfrak a \in \mathcal Z_\mathcal{H}(x)$ and a \evid{minimal atomic $\preceq$-fac\-tor\-i\-za\-tion} of $x$ is an $\mathscr A(\mathcal H)$-word $\mathfrak a \in \mathcal Z_\mathcal{H}^\mathsf{m}(x)$. Then we define 
\[
\mathcal Z_\mathcal{H}(x\,; \mathscr A(\mathcal H)) := \mathcal Z_\mathcal{H}(x) \cap \mathscr F(\mathscr A(\mathcal H))
\quad\text{and}\quad
\mathcal Z_\mathcal{H}^\mathsf{m}(x\,; \mathscr A(\mathcal H)) := \mathcal Z_\mathcal{H}^\mathsf{m}(x) \cap \mathscr F(\mathscr A(\mathcal H)), 
\]
and we let the \evid{set of atomic lengths} and the \evid{set of minimal atomic lengths} of $x$ be, resp., the sets
\[
\mathsf{L}_{\mathcal H}(x\,;\mathscr A(\mathcal H)) := \bigl\{ \|\mathfrak{a}\|_H: \mathfrak{a} \in \mathcal{Z}_{\mathcal H}(x\,;\mathscr A(\mathcal H)) \bigr\}
\quad\text{and}\quad
\mathsf{L}_{\mathcal H}^{\sf m}(x\,;\mathscr A(\mathcal H)) := \bigl\{ \|\mathfrak{a}\|_H: \mathfrak{a} \in \mathcal{Z}_{\mathcal H}^{\sf m}(x\,;\mathscr A(\mathcal H)) \bigr\}.
\]
\end{enumerate*}

\begin{enumerate*}[label=\textup{(\arabic{*})}, resume, mode=unboxed]
\item\label{def:factorizations(4)}
The premonoid $\mathcal H$ is \evid{factorable} if every $\preceq$-non-unit has at least one $\preceq$-factorization; \evid{BF-factorable} (resp., \evid{BmF-factorable}) if the set of lengths (resp., of minimal lengths) of each $\preceq$-non-unit is finite and non-empty; \evid{HF-factorable} (resp., \evid{HmF-factorable}) if the same sets of lengths (resp., of minimal lengths) are all singletons; \evid{FF-factorable} (resp., \evid{FmF-factorable}) if the quotient of $\mathcal Z_H(x)$ (resp., of $\mathcal Z_H^\mathsf{m}(x)$) by the relation of $\sqeq_\mathcal{H}$-equivalence (properly restricted) is finite and non-empty for every $\preceq$-non-unit $x$; and \evid{UF-factorable} (resp., \evid{UmF-factorable}) if the same quotients are all singletons.
\end{enumerate*}

\vskip 0.05cm

\begin{enumerate*}[label=\textup{(\arabic{*})}, resume, mode=unboxed]
\item\label{def:factorizations(5)}
In a similar way, $\mathcal H$ is \evid{atomic} if every $\preceq$-non-unit has an atomic $\preceq$-factorization; and is \evid{BF-atomic}, \evid{FmF-atomic}, etc., if the
analogous definitions given in item \ref{def:factorizations(4)} are reformulated in terms of $\preceq$-atoms.
\end{enumerate*}

\vskip 0.05cm

\begin{enumerate*}[label=\textup{(\arabic{*})}, resume, mode=unboxed]
\item\label{def:factorizations(6)}
Finally, we say that the monoid $H$ is factorable, \textup{BF}-factorable, etc., or atomic, \textup{BF}-atomic, etc., if the premonoid $(H, \mid_H)$ is, resp., factorable, \textup{BF}-factorable, etc., or atomic, \textup{BF}-atomic, etc. Furthermore, we denote by $\mathscr{I}(H)$ and $\mathscr{A}(H)$, resp., the set of ir\-red\-u\-ci\-bles and the set of $\mid_H$-atoms of $H$.
\end{enumerate*}
\end{definition}

In the next remark we examine the interrelationship among the notions from the last definition, and then we work out a couple of examples (see also Remark \ref{rem:4.14}).

\begin{remarks}\label{rem:diagram}
\begin{enumerate*}[label=\textup{(\arabic{*})}, mode=unboxed]
\item\label{rem:diagram(1)}
In the notation of Definition \ref{def:factorizations}, it is evident that, for every $x \in H$, $\mathcal{Z}^{\rm m}_\mathcal{H}(x) $ is contained in $\mathcal{Z}_\mathcal{H}(x)$ and hence $\mathsf{L}^{\rm m}_\mathcal{H}(x)$ is contained in $\mathsf{L}_\mathcal{H}(x)$. Moreover, $\mathcal{Z}^{\rm m}_\mathcal{H}(x)$ is empty if and only if so is $\mathcal{Z}_\mathcal{H}(x)$: In particular, if $\mathcal{Z}_\mathcal{H}(x)$ is non-empty, then the artinianity of the preorder $\sqeq_\mathcal{H}$ implies that $\mathcal{Z}_\mathcal{H}(x)$ has a $\sqeq_\mathcal{H}$-minimal element $\mathfrak a$ (see, e.g., \cite[Remark 3.9(3)]{Tr20(c)}) and hence $\mathfrak a \in \mathcal{Z}^{\rm m}_\mathcal{H}(x)$. It follows that $\mathcal{H}$ is fac\-tor\-able if and only if every $\preceq$-non-unit admits a minimal $\preceq$-factorization.
\end{enumerate*}

\vskip 0.05cm

\begin{enumerate*}[label=\textup{(\arabic{*})}, mode=unboxed, resume]
\item\label{rem:diagram(2)} We have already noted that every $\preceq$-atom of a premonoid $\mathcal H = (H, \preceq)$ is a $\preceq$-irreducible; whence $\mathcal{H}$ is atomic only if it is factorable. On the other hand, two $H$-words $\mathfrak a$ and $\mathfrak b$ are $\sqeq_\mathcal{H}$-equivalent (i.e., $\mathfrak a \sqeq_\mathcal{H} \allowbreak \mathfrak b \sqeq_\mathcal{H} \mathfrak a$) only if $\| \mathfrak a\|_H=\|\mathfrak b\|_H$. Therefore, if $\mathcal{H}$ is UF-factorable (resp., UmF-factorable), then it is HF-factorable (resp., HmF-factorable); and if $\mathcal{H}$ is FF-factorable (resp., FmF-factorable), then it is BF-factorable (resp., BmF-factorable). It is also clear from the definitions that $\mathcal{H}$ is BF-factorable (resp., BmF-factorable) whenever it is HF-factorable (resp., HmF-factorable); and is FF-factorable (resp., FmF-factorable) whenever it is UF-factorable (resp., UmF-factorable). Further, we have from item \ref{rem:diagram(1)} that if $\mathcal H$ is UF-, FF-, HF-, or BF-factorable, then it is UmF-, FmF-, HmF, or BmF-factorable, resp.
Lastly, all these statements remain true (essentially by the same arguments) with ``factorable'' replaced by ``atomic''. \\

\indent{}We summarize the above conclusions in the following diagram, which is ultimately a refinement (and a generalization) of analogous diagrams that are often encountered in the literature:
\end{enumerate*}

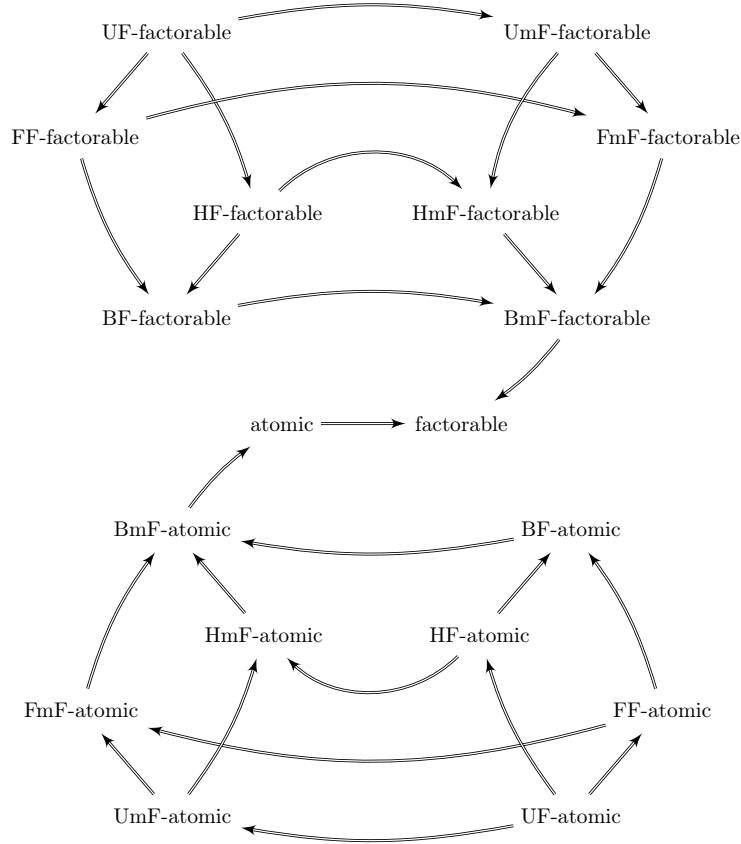
\begin{figure}[!h]
\scalebox{0.8}{
\begin{tikzpicture}[shorten >=1pt, auto, node distance=1.75cm, main node/.style={rectangle, rounded corners, minimum width=10mm, minimum height=7mm}]

  \node[main node] (UF) {UF-factorable};
  \node[main node] (FF) [below of=UF, xshift=-1.5cm] {FF-factorable};
  \node[main node] (HF) [below of=UF, xshift=1.5cm, yshift=-1.25cm] {HF-factorable};
  \node[main node] (BF) [below of=HF, xshift=-1.5cm] {BF-factorable};
  \node[main node] (A) [below of=BF, xshift=1.9cm] {atomic};
  
  \node[main node] (UmF) [right of=UF, xshift=5cm] {UmF-factorable};
  \node[main node] (HmF) [below of=UmF, xshift=-1.5cm, yshift=-1.25cm] {HmF-factorable};
  \node[main node] (FmF) [below of=UmF, xshift=1.5cm] {FmF-factorable};
  \node[main node] (BmF) [below of=HmF, xshift=1.5cm] {BmF-factorable};
  \node[main node] (F) [below of=BmF, xshift=-1.9cm] {factorable};
  
  \node[main node] (BmFa) [below of=A, xshift=-1.8cm] {BmF-atomic};
  \node[main node] (FmFa) [below of=BmFa, xshift=-1.5cm, yshift=-1.25cm] {FmF-atomic};
  \node[main node] (HmFa) [below of=BmFa, xshift=1.5cm] {HmF-atomic};
  \node[main node] (UmFa) [below of=FmFa, xshift=1.5cm] {UmF-atomic};
  
  \node[main node] (BFa) [below of=F, xshift=1.8cm] {BF-atomic};
  \node[main node] (FFa) [below of=BFa, xshift=1.5cm, yshift=-1.25cm] {FF-atomic};
  \node[main node] (HFa) [below of=BFa, xshift=-1.5cm] {HF-atomic};
  \node[main node] (UFa) [below of=FFa, xshift=-1.5cm] {UF-atomic};

  \path[every node/.style={font=\sffamily\small,
  		inner sep=1pt}]

    (UF) edge [bend right=0, double, -latex'] node[right=1mm] {} (FF)
         edge [bend left=10, double, -latex'] node[right=1mm] {} (HF)
         edge [bend left=10, double, -latex'] node[right=1mm] {} (UmF)
    (FF) edge [bend left=15, double, -latex'] node[right=1mm] {} (FmF)
         edge [bend right=10, double, -latex'] node[right=1mm] {} (BF)
    (HF) edge [bend left=45, double, -latex'] node[right=1mm] {} (HmF)
         edge [bend left=0, double, -latex'] node[right=1mm] {} (BF)
    (BF) edge [bend left=10, double, -latex'] node[right=1mm] {} (BmF)  
         
   (UmF) edge [bend right=0, double, -latex'] node[right=1mm] {} (FmF)
         edge [bend right=15, double, -latex'] node[right=1mm] {} (HmF)
   (FmF) edge [bend left=10, double, -latex'] node[right=1mm] {} (BmF)
   (HmF) edge [bend left=0, double, -latex'] node[right=1mm] {} (BmF)
   
     (A) edge [bend left=0, double, -latex'] node[right=1mm] {} (F)
   (BmF) edge [bend left=8, double, -latex'] node[right=1mm] {} (F)
  (BmFa) edge [bend left=8, double, -latex'] node[right=1mm] {} (A)
  
   (BFa) edge [bend left=10, double, -latex'] node[right=1mm] {} (BmFa)
   (FFa) edge [bend right=10, double, -latex'] node[right=1mm] {} (BFa)
         edge [bend left=15, double, -latex'] node[right=1mm] {} (FmFa)
   (HFa) edge [bend left=0, double, -latex'] node[right=1mm] {} (BFa)
         edge [bend left=45, double, -latex'] node[right=1mm] {} (HmFa)
   (UFa) edge [bend left=0, double, -latex'] node[right=1mm] {} (FFa)
         edge [bend left=10, double, -latex'] node[right=1mm] {} (HFa)
         edge [bend left=10, double, -latex'] node[right=1mm] {} (UmFa)
   
   (FmFa) edge [bend left=10, double, -latex'] node[right=1mm] {} (BmFa)
   (HmFa) edge [bend left=0, double, -latex'] node[right=1mm] {} (BmFa)
   (UmFa) edge [bend left=0, double, -latex'] node[right=1mm] {} (FmFa)
          edge [bend right=10, double, -latex'] node[right=1mm] {} (HmFa);
\end{tikzpicture}
}
\end{figure}
Below in Example \ref{exa:properties-of-factorizations}\ref{exa:properties-of-factorizations(1)}, we construct a UmF-atomic commutative monoid that is not BF-atomic and whose irreducibles are all atoms; and in Example \ref{exa:properties-of-factorizations}\ref{exa:properties-of-factorizations(2)}, we discuss a UmF-factorable commutative monoid that is not atomic. On the other hand, one can already check in the classical case of commutative domains \cite[Sect.~1.2]{GeHK04} that there exist (i) FF-atomic monoids that are not HF-atomic, and conversely; (ii) BF-atomic monoids that are neither FF- nor HF-atomic; and (iii) atomic monoids that are not BF-a\-tom\-ic. Considering that, in a unit-cancellative commutative monoid $H$, every $\mid_H$-factorization is a minimal atomic $\mid_H$-factorization (see the proof of Corollary \ref{cor:comm-unit.canc-weakly.lfgu-is-FFatom}), we can thus conclude that, in general, none of the implications in the above diagram is reversible.
\end{remarks}

\begin{examples}\label{exa:properties-of-factorizations}
\begin{enumerate*}[label=\textup{(\arabic{*})}, mode=unboxed]
\item\label{exa:properties-of-factorizations(1)}
Let $H$ be the multiplicative monoid of the integers modulo $p^n$, where $p \in \mathbb{N}$ is a prime and $n$ be an integer $\ge 2$. Set $\mathcal H = (H, \mid_H)$. Given $x \in \mathbb Z$, we write $\bar{x}$ for the residue class of $x$ modulo $p^n$. 
It is immediate that the units of $H$ are the residue classes modulo $p^n$ of the integers between $1$ and $p^n$ that are not divisible by $p$ (and there are precisely $p^{n-1}(p-1)$ of them); and the (ordinary) atoms of $H$ are the elements of the form $\bar{p} \, u$ with $u \in H^\times$ (note that $H$ is Dedekind-finite and hence, by Remark \ref{rem:preorders}\ref{rem:preorders(1)}, every $\mid_H$-unit is a unit and every $\mid_H$-atom is an atom). We claim that $H$ is an atomic monoid.\\

\indent{}In fact, every non-zero non-unit $x \in H$ can be uniquely written as $\bar{p}^{\,k} u$ for some $k \in \llb 1, n-1 \rrb$ and $u \in H^\times$, and this shows, in particular, that every irreducible is an atom. The $\mid_H$-factorizations of $x$ are therefore the non-empty $H$-words $\bar{p}\, u_1 \ast \cdots \ast \bar{p} \, u_k$ of length $k$ with $u_1, \ldots, u_k \in H$ and $u_1 \cdots u_k = u$ (which implies that the $u_i$'s are all units). On the other hand, the $\mid_H$-factorizations of $\bar{0}$ are all and only the length-$l$ $H$-words of the form $\bar{p} \, v_1 \ast \cdots \ast \bar{p}\, v_l$ with $l \ge n$ and $v_1, \ldots, v_l \in H^\times$. \\

\indent{}Thus, $H$ is UmF-atomic: All the minimal $\mid_H$-factorizations of a non-unit $x\in H$ are $\sqeq_\mathcal{H}$-equivalent (in particular, the minimal $\mid_H$-factorizations of $\bar{0}$ are the $H$-words $\bar{p}\, v_1 \ast \cdots \ast \bar{p}\, v_n$ with $v_1, \ldots, v_n \in H^\times$). On the other hand, $H$ is not even BF-atomic, for the set of  $\mid_H$-factorizations of $\bar{0}$ contains the $H$-word $\bar{p}^{\,\ast (n+k)}$ for every $k \in \mathbb N$ (and hence there are infinitely many of them that are pairwise $\sqeq_\mathcal{H}$-inequivalent). 
\end{enumerate*}

\vskip 0.05cm

\begin{enumerate*}[label=\textup{(\arabic{*})}, mode=unboxed, resume]
\item\label{exa:properties-of-factorizations(2)}
In the notation of Example \ref{exa:irrds-atoms-quarks}\ref{exa:irrds-atoms-quarks(2)}, it is clear that the premonoid $\mathscr P(S)$ is (i) atomic if and only if $S = \emptyset$, and (ii) factorable if and only if it is BmF-factorable, if and only if it is FmF-factorable, if and only if it is UmF-factorable, if and only if $|S| < \infty$: In particular, recall that the $\subseteq_S$-irreducibles are the one-element subsets of $S$ and hence every finite $X \subseteq S$ has a unique minimal $\subseteq_S$-factorization; and note that, if $a$ is an arbitrary element of $S$ (and hence $S$ is non-empty), then the length-$n$ word $\{a\} \ast \allowbreak \cdots \ast \allowbreak \{a\}$ is a $\subseteq_S$-factorization of $\{a\}$ for all $n \in \mathbb N^+$.
\end{enumerate*}
\end{examples}

On the whole, the present paper is about sufficient or necessary conditions for a premonoid to be factorable, BmF-factorable, etc.; and in particular for a monoid to be atomic, BF-atomic, FmF-atomic, etc. 
The literature abounds with results of this kind, but they are mainly about commutative or unit-cancellative monoids, see, e.g., \cite[Theorems 2.3, 2.4, and 3.1]{Co63}, \cite[Theorems 4 and 7]{Fl69}, \cite[Th\'eor\`eme de Structure]{Bouv74b}, \cite[Theorems 3.9, 3.11, 3.13, 4.4, and 4.9]{AnVL96}, \cite[Theorems 3.3, 3.4, and 3.6]{ChAnVLe11}, and \cite[parts (i) and (iv) of Theorem 2.28, and Corollary 2.29]{Fa-Tr18}, \cite[Theorem 2.8]{BaBaGo14}, and \cite{Be-Br-Na-Sm22}. On the other hand, much less is known for non-commutative non-unit-cancellative monoids (including non-commutative rings with zero divisors), see, e.g., \cite{An-Tr18, Tr20(c), Co-Tr-21(a)} (for an alternative notion of BF-ness, cf.~\cite[Lemma 2.2]{FaFa18}). Our starting point is the following proposition (the reader may want to review Remark \ref{rem:preorders}\ref{rem:preorders(2)} before reading further):

\begin{proposition}\label{prop:x-closedness}
Let $\mathcal{H}=(H,\preceq)$ be a premonoid and $\mathcal K = (K, \preceq_K)$ be a subpremonoid of $\mathcal H$ containing all the divisors (in $H$) of a fixed $x\in H$ (and hence $x$ itself). The following hold:

\vskip 0.05cm

\begin{enumerate*}[label=\textup{(\roman{*})}, mode=unboxed]
\item\label{prop:x-closedness(i)}
$\mathscr I_x(\mathcal H) = \mathscr I_x(\mathcal K)$ and $\mathscr A_x(\mathcal H) = \mathscr A_x(\mathcal K)$.
\end{enumerate*}

\vskip 0.05cm

\begin{enumerate*}[label=\textup{(\roman{*})}, mode=unboxed, resume]
\item\label{prop:x-closedness(ii)} $\mathcal{Z}_{\mathcal H}(x)=\mathcal{Z}_{\mathcal K}(x)$ and $\mathcal{Z}_{\mathcal H}(x; \mathscr{A}(\mathcal{H}))=\mathcal{Z}_{\mathcal K}(x; \mathscr{A}(\mathcal{K}))$.
\end{enumerate*}

\vskip 0.05cm

\begin{enumerate*}[label=\textup{(\roman{*})}, mode=unboxed, resume]
\item\label{prop:x-closedness(iii)} $\mathcal{Z}^{\rm m}_{\mathcal H}(x)=\mathcal{Z}^{\rm m}_{\mathcal K}(x)$ and $\mathcal{Z}^{\rm m}_{\mathcal H}(x; \mathscr{A}(\mathcal{H}))=\mathcal{Z}^{\rm m}_{\mathcal K}(x; \mathscr{A}(\mathcal{K}))$.
\end{enumerate*}

\vskip 0.05cm

\begin{enumerate*}[label=\textup{(\roman{*})}, mode=unboxed, resume]
\item\label{prop:x-closedness(iv)} $\mathsf{L}_{\mathcal H}(x)=\mathsf{L}_{\mathcal K}(x)$ and  $\mathsf{L}_{\mathcal H}(x; \mathscr{A}(\mathcal{H}))=\mathsf{L}_{\mathcal K}(x; \mathscr{A}(\mathcal{K}))$.
\end{enumerate*}

\vskip 0.05cm

\begin{enumerate*}[label=\textup{(\roman{*})}, mode=unboxed, resume]
\item\label{prop:x-closedness(v)} $\mathsf{L}^{\rm m}_{\mathcal H}(x)=\mathsf{L}^{\rm m}_{\mathcal K}(x)$ and  $\mathsf{L}^{\rm m}_{\mathcal H}(x; \mathscr{A}(\mathcal{H}))=\mathsf{L}^{\rm m}_{\mathcal K}(x; \mathscr{A}(\mathcal{K}))$.
\end{enumerate*}

\vskip 0.05cm

\noindent{}In particular, these conclusions are true when $K$ is a divisor-closed submonoid of $H$ containing $x$.
\end{proposition}

\begin{proof}
The ``In particular'' part of the proposition is obvious, because every divisor-closed submonoid of $H$ containing $x$ does also contain each and every divisor of $x$ in $H$. Moreover, \ref{prop:x-closedness(iv)} and \ref{prop:x-closedness(v)} are immediate consequences of \ref{prop:x-closedness(ii)} and \ref{prop:x-closedness(iii)}, resp. So, we concentrate on items \ref{prop:x-closedness(i)}--\ref{prop:x-closedness(iii)}.

\vskip 0.05cm

\ref{prop:x-closedness(i)} We will only check that $\mathscr I_x(\mathcal H) = \mathscr I_x(\mathcal K)$; showing $\mathscr A_x(\mathcal H) = \mathscr A_x(\mathcal K)$ can be done in a sim\-i\-lar (and even simpler) way, and we leave the details to the reader. 

To start with, pick $a \in \mathscr{I}_x(\mathcal{K})$ and assume for a contradiction that $a\notin \mathscr{I}_x(\mathcal{H})$. Then $a \mid_K x$ and hence $a \mid_H x$ (because $K$ is a submonoid of $H$). We thus get that $a \notin \mathscr{I}(\mathcal{H})$; and since $a$ is a $\preceq$-non-unit (recall from Remark \ref{rem:preorders}\ref{rem:preorders(2)} that $\mathcal K^\times = K \cap \mathcal H^\times$), it follows that there exist $\preceq$-non-units $u, v \in H$ with $u \prec a$ and $v \prec a$ such that $a = uv$. So, $u$ and $v$ are $\mid_H$-divisors of $x$ and hence elements of $K$, because $K$ contains all the divisors of $x$ in $H$ (by hypothesis). This is, however, in contradiction with the $\preceq_K$-irreducibility of $a$, because $u$ and $v$ are then $\preceq_K$-non-units with $u \prec_K a$ and $v \prec_K a$. In consequence, $ \mathscr{I}_x(\mathcal{K})\subseteq  \mathscr{I}_x(\mathcal{H})$. 

As for the opposite inclusion, let $a \in \mathscr{I}(\mathcal{H})$ such that $x=yaz$ for some $y,z\in H$. Since $K$ contains all the divisors of $x$ in $H$, each of $a$, $y$, and $z$ lies in $K$. So, $a$ is a $\preceq_K$-non-unit and a divisor of $x$ in $K$; in particular, if $a$ is not $\preceq_K$-irreducible, then there exist $\preceq_K$-non-units $u,v\in K$ with $u \prec_K a$ and $v \prec_K a$ such that $a=uv$, which is impossible as it means that $a$ is not $\preceq$-irreducible. All in all, we can therefore conclude that $a \in \mathscr{I}_x(\mathcal{K})$ and hence $\mathscr I_x(\mathcal H) \subseteq \mathscr{I}_x(\mathcal{K})$.
    
\ref{prop:x-closedness(ii)} Let $\mathfrak a \in \mathcal{Z}_{\mathcal H}(x)$ be a $\preceq$-factorization of $x$, meaning that $\mathfrak a$ is an  $\mathscr{I}(\mathcal{H})$-word such that $\pi_H(\mathfrak a) = x$. Then $\mathfrak a[i]\in\mathscr{I}_x(\mathcal{H})$ for each $i \in \llb 1, \|\mathfrak a\|_H \rrb$; and by item \ref{prop:x-closedness(i)}, this implies that $\mathfrak a$ is an $\mathscr{I}(\mathcal{K})$-word such that $\pi_K(\mathfrak a) = x$, that is, a $\preceq_K$-factorization of $x$. It follows that $\mathcal Z_\mathcal{H}(x)$ is contained in $\mathcal Z_\mathcal{K}(x)$, and the reverse inclusion is analogous. In consequence, we find that
\begin{equation}
\label{equ:atoms-in-certain-subpremonoids}
\begin{split}
\mathcal{Z}_{\mathcal H}(x;\mathscr{A}(\mathcal{H}))& = \mathcal{Z}_{\mathcal H}(x)\cap \mathscr{F}(\mathscr{A}(\mathcal{H}))  = \mathcal{Z}_{\mathcal H}(x)\cap \mathscr{F}(\mathscr{A}_x(\mathcal{H})) \\
& \stackrel{\ref{prop:x-closedness(i)}}{=} \mathcal{Z}_{\mathcal K}(x)\cap \mathscr{F}(\mathscr{A}_x(\mathcal{K}))= \mathcal{Z}_{\mathcal K}(x)\cap \mathscr{F}(\mathscr{A}(\mathcal{K}))= \mathcal{Z}_{\mathcal K}(x;\mathscr{A}(\mathcal{K})),
\end{split}
\end{equation}
where, among other things, we have used that if an $\mathscr{A}(\mathcal{H})$-word (resp., an $\mathscr{A}(\mathcal{K})$-word) $\mathfrak a$ is a $\preceq$-fac\-tor\-i\-za\-tion (resp., a $\preceq_K$-factorization) of $x$, then $\mathfrak a$ is an $\mathscr{A}_x(\mathcal{H})$-word (resp., an $\mathscr{A}_x(\mathcal{K})$-word).

\vskip 0.05cm

\ref{prop:x-closedness(iii)} 
Let $\mathfrak a$ and $\mathfrak b$ be $K$-words. In the notation of Definition \ref{def:shuffling-preorder}, we have that $\mathfrak a \sqeq_\mathcal{K} \mathfrak b$ if and only if there exists an injective function $\sigma: \llb 1, \|\mathfrak a\|_K \rrb \to \llb 1, \|\mathfrak b\|_K\rrb$ such that $\mathfrak a[i]\preceq_K \mathfrak b[\sigma(i)]\preceq_K \mathfrak a[i]$ for each $i \in \allowbreak \llb 1, \allowbreak \|\mathfrak a\|_K \rrb$. Since $\|\mathfrak z\|_K = \|\mathfrak z\|_H$ for every $K$-word $\mathfrak z$ and, on the other hand, $a \preceq_K b$ if and only if $a \preceq b$ and $a, b \in K$, we thus find that $\mathfrak a \sqeq_{\mathcal{K} }\mathfrak b$ if and only if $\mathfrak a \sqeq_{\mathcal{H} }\mathfrak b$.

At the end of the day, this shows that the preorder $\sqeq_\mathcal{K}$ is the restriction to $\mathscr F(K)$ of the preorder $\sqeq_\mathcal{H}$. It is therefore clear that a word $\mathfrak z$ in a set of $K$-words is $\sqeq_\mathcal{K}$-minimal if and only if it is $\sqeq_\mathcal{H}$-minimal; and by the first part of item \ref{prop:x-closedness(ii)}, we obtain that $\mathcal{Z}^{\rm m}_{\mathcal H}(x)=\mathcal{Z}^{\rm m}_{\mathcal K}(x)$ for every $x \in K$. Similarly as in the proof of Eq.~\eqref{equ:atoms-in-certain-subpremonoids}, it then follows that $\mathcal{Z}^{\rm m}_{\mathcal H}(x;\mathscr{A}(\mathcal{H})) = 
\mathcal{Z}^{\rm m}_{\mathcal K}(x;\mathscr{A}(\mathcal{K}))$.
\end{proof} 

\begin{corollary}\label{cor:divisor-closedness}
Let $\mathcal{H} = (H,\preceq)$ be a premonoid and $\mathcal K = (K, \preceq_K)$ be a subpremonoid of $\mathcal H$ such that $K$ is a divisor-closed submonoid. Then
$\mathscr{I}(\mathcal{K}) = K \cap \mathscr{I}(\mathcal{H})$ and $\mathscr{A}(\mathcal{K}) = K \cap \mathscr{A}(\mathcal{H})$.
\end{corollary}

\begin{proof}
We just prove the first equality, the second can be proved in an analogous fashion. 

To begin, pick $x \in \mathscr{I}(\mathcal{K})$. It is obvious from the definitions that $x\in K$ and $x \in \mathscr{I}_x(\mathcal{K})$. Since $K$ is divisor-closed in $H$ (and hence contains all the divisors of $x$ in $H$), we thus get from Proposition \ref{prop:x-closedness}\ref{prop:x-closedness(i)} that $x \in \mathscr{I}_x(\mathcal{H})$ and hence $x \in K \cap \allowbreak \mathscr{I}(\mathcal{H})$. It follows that $\mathscr{I}(\mathcal{K})\subseteq K\cap \mathscr{I}(\mathcal{H})$. 

As for the reverse inclusion, let $x \in K \cap \mathscr{I}(\mathcal{H})$. As before, we have $x\in \mathscr{I}_x(\mathcal{H})=\mathscr{I}_x(\mathcal{K})$, since $x$ is in $K$ and $K$ satisfies the hypotheses of Proposition \ref{prop:x-closedness}\ref{prop:x-closedness(i)}. It follows that $x\in \mathscr{I}(\mathcal{K})$, and so we are done.
\end{proof} 

For the next result we recall from Remark \ref{rem:preorders}\ref{rem:preorders(2)} that $\llb x \rrb_\mathcal{H}$ denotes the subpremonoid of a premonoid $\mathcal{H}$, whose ``ground monoid'' is the smallest divisor-closed submonoid $\llb x \rrb_H$ of $H$ containing an element $x \in H$.

\begin{corollary}\label{cor:cor2}
A premonoid $\mathcal{H} = (H, \preceq)$ is factorable, \textup{UF}-factorable, \textup{HF}-factorable, \textup{FF}-factorable, \textup{BF}-fac\-tor\-able, \textup{UmF}-factorable, \textup{HmF}-factorable, \textup{BmF}-factorable, or \textup{FmF}-factorable (resp., atomic, \textup{UF}-a\-tom\-ic, etc.) if and only if so is $\llb x \rrb_\mathcal H$ for every $\preceq$-non-unit $x$. 
\end{corollary}

\begin{proof}
We will just show that, if $\mathcal{H}$ is a factorable premonoid, then so is $\llb x \rrb_\mathcal H$ for every $\preceq$-non-unit $x$: The converse is immediate from Proposition \ref{prop:x-closedness}\ref{prop:x-closedness(ii)}, because every $\preceq$-non-unit has a $\preceq_{K}$-factorization; and the other logical equivalences (involving atomicity, BF-ness, etc.)~are proved in essentially the same way, based on Proposition \ref{prop:x-closedness}\ref{prop:x-closedness(ii)}--\ref{prop:x-closedness(v)} (we leave the details to the reader).

So, assume $\mathcal{H}$ is factorable, let $x$ be a $\preceq$-non-unit, and for ease of notation set $K := \llb x \rrb_H$ and $\mathcal{K} := \allowbreak \llb x \rrb_\mathcal H$. We claim that $\mathcal K$ is a factorable premonoid. For, fix $y \in K \setminus \mathcal{K}^\times$. By Remark \ref{rem:preorders}\ref{rem:preorders(2)}, $y$ is a $\preceq$-non-unit and hence, by assumption, $\mathcal Z_\mathcal{H}(y)$ is non-empty. Therefore, we gather from Proposition \ref{prop:x-closedness}\ref{prop:x-closedness(ii)} that $\mathcal Z_\mathcal{K}(y)$, too, is non-empty, since $K$ is a divisor-closed submonoid of $H$. By the arbitrariness of $x$ and $y$, this is enough to finish the proof.
\end{proof}

Incidentally, we get from Remark \ref{rem:preorders}\ref{rem:preorders(3)} that Corollary \ref{cor:divisor-closedness} and items \ref{prop:x-closedness(ii)}--\ref{prop:x-closedness(v)} of Proposition \ref{prop:x-closedness} (specialized to the divisibility preorder) are a generalization of \cite[Proposition 2.21]{Fa-Tr18} and \cite[Proposition 4.7]{An-Tr18} for the part concerning factorizations and sets of lengths relative to (ordinary) atoms. Note in this regard that the existence itself of an atom in a monoid $H$ is a sufficient condition for $H$ to be Dedekind-finite \cite[Proposition 2.30]{Fa-Tr18}, with the result that the set of $\mid_H$-units is nothing else than the set $H^\times$ of (ordinary) units and hence the set of $\mid_H$-atoms is the set of atoms (Remark \ref{rem:preorders}\ref{rem:preorders(1)}).

\section{Finitely generated monoids and beyond}
\label{sec:beyond-fg-mons}
In this section, we focus attention on premonoids that are, in a sense, ``arithmetically small''. The basic idea is formalized in the next definition, which contains some of the main novelties of the paper.

\begin{definition}\label{def:lfgu}
\begin{enumerate*}[label=\textup{(\arabic{*})}, mode=unboxed]
\item\label{def:lfgu(1)}
Given a premonoid $\mathcal H = (H, \preceq)$, we denote by $\llangle x \rrangle_H$ the submonoid of $H$ generated by the divisors of an element $x \in H$ and by $\llangle x\rrangle_\mathcal H$ the subpremonoid of $\mathcal H$ obtained by endowing $\llangle x\rrangle_H$ with the restriction of the preorder $\preceq$. In particular, we call $\llangle x \rrangle_\mathcal H$ the \evid{germ of $\mathcal H$ at $x$}.
\end{enumerate*}

\vskip 0.05cm

\begin{enumerate*}[label=\textup{(\arabic{*})}, mode=unboxed, resume]
\item\label{def:lfgu(2)}
The premonoid $\mathcal H$ is \evid{finitely generated} or \evid{f.g.}~if $H$ is a finitely generated monoid, i.e., $H = \langle A \rangle_H$ for a finite $A \subseteq H$; \evid{finitely generated up to units} or \evid{f.g.u.}~if there exists a finite $A \subseteq H$ such that $H = \allowbreak \langle \mathcal H^\times A\, \mathcal H^\times \rangle_H$; \evid{locally f.g.u.}~or \textsf{l.f.g.u.}~(resp., \evid{locally f.g.}~or \evid{l.f.g.})~if, for each $\preceq$-non-unit $x$, the premonoid $\llb x \rrb_\mathcal{H}$ is f.g.u.~(resp., f.g.); and \evid{weakly l.f.g.u.}~(resp., \evid{weakly l.f.g.})~if the germ of $\mathcal H$ at every $\preceq$-non-unit is an f.g.u.~(resp., f.g.)~premonoid.
\end{enumerate*}

\vskip 0.05cm

\begin{enumerate*}[label=\textup{(\arabic{*})}, resume, mode=unboxed]
\item\label{def:lfgu(3)}
On the other hand, $\mathcal H$ is \evid{of finite type} or \evid{oft} if there is a finite set $A \subseteq \mathscr{I}(\mathcal{H})$ such that every $\preceq$-fac\-tor\-i\-za\-tion of a $\preceq$-non-unit is $\sqeq_\mathcal{H}$-equivalent to an $A$-word; and is \evid{locally of finite type} or \evid{loft} if, for each $\preceq$-non-unit $x$, there is a finite set $A_x \subseteq \mathscr{I}(\mathcal{H})$ such that every $\preceq$-factorization of $x$ is $\sqeq_\mathcal{H}$-e\-quiv\-a\-lent to an $A_x$-word.
\end{enumerate*}

\vskip 0.05cm

\begin{enumerate*}[label=\textup{(\arabic{*})}, resume, mode=unboxed]
\item\label{def:lfgu(4)}
Finally, we say that the monoid $H$ is [weakly] l.f.g., f.g.u., [weakly] l.f.g.u., oft, or loft if $(H, \mid_H)$ is, resp., a [weakly] l.f.g., f.g.u., [weakly] l.f.g.u., oft, or loft premonoid.
\end{enumerate*}
\end{definition}

In particular, the notion of l.f.g.u.~monoid is ultimately a generalization of \cite[Definition 2.7.6.5]{GeHK06} from commutative, cancellative to arbitrary monoids (see Remark \ref{rem:lfgu-monoids}\ref{rem:lfgu-monoids(2)} for further details).

\begin{lemma}
\label{lem:germs-of-fgu-premonos-are-fgu}
Every germ of an f.g.u.~premonoid is itself an f.g.u.~premonoid.
\end{lemma}

\begin{proof}
Let $\mathcal H = (H, \preceq)$ be an f.g.u.~premonoid and pick $x \in H$. There then exists a finite $A \subseteq H$ such that $H = \langle \mathcal H^\times A\, \mathcal H^\times \rangle_H$. Denote by $A_x$ the set of all $a \in A$ with the property that $a \mid_H \allowbreak y$ for some divisor $y$ of $x$ in $H$ and by $\mathcal K = (K, \preceq_K)$ the germ of $\mathcal H$ at $x$ (so that $K$ is the submonoid $\llangle x \rrangle_H$ of $H$ generated by the divisors of $x$ in $H$). It is obvious that $A_x$ is a finite subset of $K$, since it is a subset of $A$ and every $a \in A_x$ is a divisor of $x$ in $H$ (by transitivity of $\mid_H$). It follows that $\mathcal K^\times A_x \mathcal K^\times \subseteq K$, and we claim $K \subseteq \allowbreak \langle \mathcal K^\times A_x \mathcal K^\times \rangle_K$: This will show that $\mathcal K$ is an f.g.u.~premonoid and finish the proof.

Fix $y \in H$ such that $y \mid_H x$. Since $K$ is generated by the divisors of $x$ in $H$, it suffices to check that $y \in \langle \mathcal K^\times A_x \mathcal K^\times \rangle_K$. For, note that $H = \langle \mathcal H^\times A\, \mathcal H^\times \rangle_H$ implies the existence of $u_1, v_1, \ldots, u_n, v_n \in \mathcal H^\times$ and $a_1, \ldots, a_n \in A$ such that $y = u_1 a_1 v_1 \cdots u_n a_n v_n$; and since each of the factors on the right is a divisor of a divisor of $x$ in $H$, we have (by the definitions of $K$ and $A_x$) that $u_i, v_i \in K \cap \mathcal H^\times$ and $a_i \in A_x$ for each $i \in \llb 1, n \rrb$. This is enough to conclude, as we know from Remark \ref{rem:preorders}\ref{rem:preorders(2)} that $K \cap \mathcal H^\times = \mathcal K^\times$.
\end{proof}

\begin{proposition}\label{pro:f.g.u. is l.f.g.u. is w.l.f.g.u.}
For a premonoid $\mathcal H = (H, \preceq)$, all implications in the following diagram hold:
\begin{figure}[!h]
\scalebox{0.9}{
\begin{tikzpicture}[shorten >=1pt, auto, node distance=1.75cm, main node/.style={rectangle, rounded corners, minimum width=10mm, minimum height=7mm}]

  \node[main node] (fg) {f.g.};
  \node[main node] (lfg) [right of=fg, xshift=2cm] {l.f.g.};
  \node[main node] (wlfg) [right of=lfg, xshift=2cm] {weakly l.f.g.};
  
  \node[main node] (fgu) [below of=fg] {f.g.u.};
  \node[main node] (lfgu) [right of=fgu, xshift=2cm] {l.f.g.u.};
  \node[main node] (wlfgu) [right of=lfgu, xshift=2cm] {weakly l.f.g.u.};

  \path[every node/.style={font=\sffamily\small,
  		inner sep=1pt}]

    (fg) edge [bend right=0, double, -latex'] node[midway, above=0.25mm] {} (lfg)
        edge [bend left=0, double, -latex'] node {} (fgu)
    (fgu) edge [bend left=0, double, -latex'] node[midway, below=0.5mm] {} (lfgu)
    (lfg) edge [bend left=0, double, -latex'] node {} (wlfg)
        edge [bend left=0, double, -latex'] node {} (lfgu)
    (lfgu) edge [bend left=0, double, -latex'] node {} (wlfgu)
    (wlfg) edge [bend left=0, double, -latex'] node {} (wlfgu);
\end{tikzpicture}
}
\end{figure}

\noindent{}In particular, every f.g.u.~monoid is l.f.g.u.~and every l.f.g.u.~monoid is weakly l.f.g.u.
\end{proposition}

\begin{proof}
The implications corresponding to the vertical arrows in the above diagram are obvious. Moreover, showing that an f.g.~premonoid is l.f.g.~and an l.f.g.~premonoid is weakly l.f.g., is pretty much the same as (and in fact easier than) showing that an f.g.u.~premonoid is l.f.g.u.~and an l.f.g.u.~premonoid is weakly l.f.g.u. Lastly, proving that an f.g.u.~premonoid $\mathcal H$ is l.f.g.u., is tantamount to proving that the premonoid $\llb x \rrb_\mathcal{H}$ is f.g.u.~for each $\preceq$-non-unit $x$; and this can be done in a similar way as to Lemma \ref{lem:germs-of-fgu-premonos-are-fgu}, with the set $A_x$ now equal to $A \cap \llb x \rrb_H$ (we leave the details to the reader). Therefore, it only remains to check that an l.f.g.u.~premonoid is weakly l.f.g.u.

For, assume $\mathcal H$ is l.f.g.u.~and fix a $\preceq$-non-unit $x \in H$. Then $\llb x \rrb_{\mathcal{H}}$ is an f.g.u.~premonoid and, by Lemma \ref{lem:germs-of-fgu-premonos-are-fgu}, so is every germ of $\llb x \rrb_{\mathcal{H}}$. It follows that the germ $\llangle x \rrangle_{\mathcal{H}}$ of $\mathcal H$ at $x$ is an f.g.u.~premonoid, because the submonoid $\llangle x \rrangle_H$ of $H$ generated by the divisors of $x$ in $H$ is obviously a submonoid of the smallest divisor-closed submonoid $\llb x \rrb_H$ of $H$ containing $x$ and, on the other hand, it is routine to check that the restriction of the preorder $\preceq$ to $\llangle x \rrangle_H$ is nothing different from the restriction to $\llangle x \rrangle_H$ of the restriction of $\preceq$ to $\llb x \rrb_H$ (with the result that $\llangle x \rrangle_{\mathcal{H}}$ is also the germ of $\llb x \rrb_{\mathcal{H}}$ at $x$). This finishes the proof, since $x$ was an arbitrary $\preceq$-non-unit.
\end{proof}

In the remainder of the section, we will show that not only weakly l.f.g.u.~premonoids are a generalization of l.f.g.u.~monoids (Proposition \ref{pro:f.g.u. is l.f.g.u. is w.l.f.g.u.}), but they are also well behaved with respect to certain properties of interest. First, though, some remarks and examples are in order:

\begin{remarks}\label{rem:lfgu-monoids}
\begin{enumerate*}[label=\textup{(\arabic{*})}, mode=unboxed]
\item\label{rem:lfgu-monoids(1)}
A premonoid $\mathcal H = (H, \preceq)$ is f.g.u.~if and only if there is a finite $A \subseteq H$ such that every $\preceq$-non-unit lies in the submonoid of $H$ generated by $\mathcal H^\times A\, \mathcal H^\times $: In particular, the latter condition implies that $H = \langle \mathcal H^\times \bar{A} \, \mathcal H^\times \rangle_H$, where $\bar{A} := A \cup \{1_H\} \subseteq H$ is still finite (the other direction is obvious).
\end{enumerate*}

\vskip 0.05cm

\begin{enumerate*}[label=\textup{(\arabic{*})}, mode=unboxed, resume]
\item\label{rem:lfgu-monoids(2)}
A Dedekind-finite reduced monoid is l.f.g.u.~if and only if it is l.f.g.; and a commutative monoid is f.g.u.~(resp., l.f.g.u.)~if and only if modding out the group of units gives an f.g.~(resp., l.f.g.)~monoid, cf.~\cite[Propositions 2.7.4.1 and 2.7.8.1]{GeHK06}. 
Moreover, it follows from \cite[Proposition 1.1.7]{GeHK06} that a cancellative, commutative monoid $H$ is l.f.g.u.~if and only if, for every $x \in H$, the set of (ordinary) atoms $a \in H$ such that $a \mid_H x$ is finite up to associates; and from \cite[Proposition 2.7.8.4]{GeHK06} that every cancellative, commutative, l.f.g.u.~monoid is FF-atomic. This last result will be generalized in Corollary \ref{cor:comm-unit.canc-weakly.lfgu-is-FFatom}.
\end{enumerate*}

\vskip 0.05cm

\begin{enumerate*}[label=\textup{(\arabic{*})}, mode=unboxed, resume]
\item\label{rem:lfgu-monoids(3)}
We say that a monoid $H$ is \textsf{idf} (for  irreducible-divisor-finite) if, for every $\mid_H$-non-unit $x$, the set $\mathscr I_x(H)$ of irreducible divisors of $x$ is finite up to $\mid_H$-equivalence. This is ultimately a generalization of (commutative) idf-domains in the sense of A.~Grams and H.~Warner \cite{Gra-War75}, which are in turn a generalization of Cohen-Kaplansky domains, i.e., idf-domains where each non-zero non-unit is a product of atoms \cite{CoKa46, AndMo90}. \\

\indent{}We claim that every idf monoid $H$ is loft. For, let $x$ be a $\mid_H$-non-unit, $a_1, \ldots, a_n$ be representatives of the (finitely many) $\mid_H$-equivalence classes in $\mathscr I_x(H)$, and $b_1 \ast \dots \ast b_k$ be a non-empty $\mid_H$-factorization of $x$ of length $k$. Then $b_i \in \mathscr I_x(H)$ for each $i \in \llb 1, k \rrb$ and there is an index $j_i \in \llb 1, n\rrb$ such that $b_i$ is $\mid_H$-equivalent to $a_{j_i}$. It is then clear from Definition \ref{def:shuffling-preorder} that $b_1\ast \dots \ast b_k$ is $\sqeq_\mathcal{H}$-equivalent to the $A_x$-word $a_{j_1}\ast \dots \ast a_{j_k}$, where $\mathcal H := (H, \mid_H)$ and  $A_x := \{a_1, \ldots, a_n\}$. Thus $H$ is loft (by the arbitrariness of $x$).
\end{enumerate*}
\end{remarks}

\begin{examples}\label{exa:lfgu-premonoids}
\begin{enumerate*}[label=\textup{(\arabic{*})},mode=unboxed]
\item\label{exa:lfgu-premonoids(1)}
A list of cancellative, commutative, l.f.g.u.~monoids (and, in particular, of commutative domains whose non-zero elements form an l.f.g.u.~monoid under multiplication) can be found in \cite[Example 3.4]{Tr19a} and \cite[Example 2.1]{Zh19}. The list includes Cohen-Kaplansky domains (Remark \ref{rem:lfgu-monoids}\ref{rem:lfgu-monoids(3)}), Krull monoids \cite[Definition 2.3.1.5]{GeHK06}, Krull (and especially Dedekind) domains \cite[Definition 2.10.1.1]{GeHK06}, rings of integer-valued polynomials over a unique factorization domain \cite{Rei14}, numerical monoids \cite{As-GaSa16}, etc.
\end{enumerate*}

\vskip 0.05cm

\begin{enumerate*}[label=\textup{(\arabic{*})},mode=unboxed, resume]
\item\label{exa:lfgu-premonoids(2a)} Fix $k \in \mathbb N^+$. Every submonoid of $(\mathbb N^{\times k}, +)$ is weakly l.f.g., because $(y_1, \ldots, y_k)$ divides $(x_1, \ldots, x_k)$ in $(\mathbb N^{\times k}, +)$ only if $y_i \in \llb 0, x_i \rrb$ for each $i \in \mathbb N$. However, the submonoid $H$ of $(\mathbb N^{\times 2}, +)$ generated by the set $A := \bigcup_{n \ge 1} \{(1, n), (n,1)\}$ is not l.f.g.u.: In fact, $H$ is a reduced monoid with $\mathscr A(H) = A$ (implying that $H$ is not f.g.u.); and since $n\,(1,1) = (1,n) + (n,1)$ for all $n \in \mathbb N^+$, we have $\llb (1,1) \rrb_H = H$.
\end{enumerate*}

\vskip 0.05cm

\begin{enumerate*}[label=\textup{(\arabic{*})},mode=unboxed, resume]
\item\label{exa:lfgu-premonoids(2)} Following \cite[Sect.~3]{Fa-Tr18}, let $\mathcal P_{\mathrm{fin},1}(H)$ be the \evid{reduced power monoid} of a monoid $H$, i.e., the monoid obtained by endowing the family of all \emph{finite} subsets of $H$ containing the identity $1_H$ with the binary operation of \evid{setwise multiplication} $(X, Y) \mapsto XY$.
We gather from \cite[Proposition 4.11(i)]{Tr20(c)} that $\mathcal P_{\mathrm{fin},1}(H)$ is a Dedekind-finite reduced monoid; and this implies by Example \ref{exa:preord-mons}\ref{exa:preord-mons(1)} that, endowed with the divisibility preorder, $\mathcal P_{\mathrm{fin},1}(H)$ is a weakly positive monoid. Furthermore, it is easily checked that $\mathcal P_{\mathrm{fin},1}(H)$ is loft and weakly l.f.g.u.: Both of these properties follow at once from considering that, if $Y$ is a divisor of $X$ in $\mathcal P_{\mathrm{fin},1}(H)$, then $Y$ is contained in $X$ and hence $X$ has finitely many divisors (in fact, no more than $2^{|X|-1}$ of them), because $Y = 1_H Y 1_H \subseteq UYV$ for all $U, V \in \mathcal P_{\mathrm{fin},1}(H)$. \\

\indent{}On the other hand, $\mathcal P_{\mathrm{fin},1}(H)$ need not be an l.f.g.u.~monoid. For, denote by $\mathcal P_{\mathrm{fin},0}(\mathbb N)$ the reduced power monoid of the additive monoid of non-negative integers and write the operation of $\mathcal P_{\mathrm{fin},0}(\mathbb N)$ additively. For every finite set $X \subseteq \mathbb N$ with $0 \in X$, there is an integer $n \ge 1$ such that $n\,\{0, 1\} = X + Y$ for some $Y \in \allowbreak \mathcal P_{\mathrm{fin},0}(\mathbb N)$: It suffices to take $n = 2\max(X)$ and $Y = \llb 0, \max(X) \rrb$. Thus, the smallest divisor-closed submonoid of $\mathcal P_{\mathrm{fin},0}(\mathbb N)$ containing $\{0, 1\}$ is $\mathcal P_{\mathrm{fin},0}(\mathbb N)$ itself; and since $\mathcal P_{\mathrm{fin},0}(\mathbb N)$ is reduced and has infinitely many atoms (for instance, we get from \cite[Proposition 4.1(iv)]{Fa-Tr18} that any set of the form $\{0, a\}$ with $a \in \allowbreak \mathbb N^+$ is an atom), we conclude that the monoid is not l.f.g.u. \\

\indent{}Lastly, we note that $\mathcal P_{\mathrm{fin},1}(H)$ is finite if and only if so is $H$, which gives a class of finite monoids with a rich arithmetic already in the case when $H$ is a finite cyclic group \cite[Sect.~5]{An-Tr18}. Moreover, the $\mid_{\mathcal P_{\mathrm{fin},1}(H)}$-irreducibles of $\mathcal P_{\mathrm{fin},1}(H)$ are not necessarily atoms (whereas the $\mid_{\mathcal P_{\mathrm{fin},1}(H)}$-atoms are atoms, since $\mathcal P_{\mathrm{fin},1}(H)$ is Dedekind-finite): In fact, it is known from \cite[Proposition 4.11(iii)]{Tr20(c)} that every $\mid_{\mathcal P_{\mathrm{fin},1}(H)}$-irreducible is a $\mid_{\mathcal P_{\mathrm{fin},1}(H)}$-quark (and vice versa); and from \cite[Theorem 4.12]{Tr20(c)} that every $\mid_{\mathcal P_{\mathrm{fin},1}(H)}$-irreducible is an atom if and only if $1_H \ne x^2 \ne x$ for each non-identity $x \in H$.
\end{enumerate*}

\vskip 0.05cm

\begin{enumerate*}[label=\textup{(\arabic{*})},mode=unboxed, resume]
\item\label{exa:lfgu-premonoids(3)}
Let $G_0$ be a subset of a (multiplicatively written) group $G$ and $\mathscr F_{\rm ab}(G)$ be the \evid{free abelian monoid} on $G$, namely, the quotient of the free monoid $\mathscr F(G)$ by the monoid congruence $\equiv$ that identifies two $G$-words if and only if they can be obtained from each other by a permutation of their letters. 
We will use $\bar{\mathfrak u}$ for the (congruence) class in $\mathscr F_{\rm ab}(G)$ of a $G$-word $\mathfrak u$ and write $\mathscr F_{\rm ab}(G)$ multiplicatively.\\

\indent{}Consider the function $\Pi_G \colon \mathscr F_{\rm ab}(G) \to G$ that maps the class of a $G$-word to its direct image under the factorization homomorphism $\pi_G \colon \mathscr F(G) \to G$ of $G$ (note that $\Pi_G$ is well defined, because if $\mathfrak u \equiv \mathfrak v$ then $\bar{\mathfrak u} = \bar{\mathfrak v}$). 
The set of all classes $\bar{\mathfrak u}$ with $\mathfrak u \in \mathscr F(G_0)$ and $1_G \in \Pi_G(\mathfrak u)$ forms then a submonoid of $\mathscr F_{\rm ab}(G)$, called the \evid{monoid of product-one sequences over $G$ with support in $G_0$} and herein denoted by $\mathcal B(G_0)$: This is a monoid with a central role in the classical theory of factorization since it can often be used as a ``canonical model'' for much more complicated objects (see \cite{Fad-Zho22} and references therein).\\

\indent{}It is fairly obvious that $\mathcal B(G_0)$ is a cancellative, commutative, weakly l.f.g.u.~monoid with trivial group of units: In particular, if $\mathfrak u$ and $\mathfrak v$ are $G_0$-words with $\bar{\mathfrak u} \mid_{\mathcal B(G_0)} \bar{\mathfrak v}$, then $\mathfrak u$ is a \evid{scrambled subword} of $\mathfrak v$, i.e., there exists an injective function $\sigma \colon \llb 1, \| \mathfrak u \|_G \rrb \to \llb 1, \| \mathfrak v \|_G \rrb$ such that $\mathfrak u[i] = \mathfrak v[\sigma(i)]$ for each $i \in \llb 1, \| \mathfrak u \|_G \rrb$; and there are, of course, finitely many scrambled subwords of $\mathfrak v$. Yet, $\mathcal B(G)$ need not be an l.f.g.u.~monoid. \\

\indent{}For, let $G$ be a dihedral group of infinite order with generators $\alpha$ and $\tau$ such that $\tau^2 = 1_G$ and $\alpha \,\tau = \allowbreak \tau \alpha^{-1}$. The $G$-word $\mathfrak u_n := \alpha^{\ast 2n} \ast \tau^{\ast 2}$ rep\-re\-sents a product-one sequence $\bar{\mathfrak u}_n$ over $G$ for each $n \in \mathbb N^+$, because $(\alpha^n \, \tau)^2 = \alpha^n \tau^2 \alpha^{-n} = 1_G$; and it is routine to check that $\bar{\mathfrak u}_n$ is in fact an atom of $\mathcal B(G)$. It follows that $\llb \bar{u}_1 \rrb_{\mathcal B(G)}$ is not an f.g.u.~monoid, since $\bar{\mathfrak u}_n \mid_{\mathcal B(G)} \bar{\mathfrak u}_1^{\,n}$ for all $n \in \mathbb N^+$. So, $\mathcal B(G)$ is not l.f.g.u.
\end{enumerate*}

\vskip 0.05cm

\begin{enumerate*}[label=\textup{(\arabic{*})},mode=unboxed, resume]
\item\label{exa:lfgu-premonoids(4)}
Let $R$ be a commutative PID and $\mathcal M_n(R)$ be the ring of $n$-by-$n$ matrices over $R$ (we refer to \cite{Jac85} for basic aspects of ring theory). It is well known that the set of matrices $A \in \mathcal M_n(R)$ with non-zero de\-ter\-mi\-nant is a cancellative submonoid, herein denoted by $H$, of the multiplicative monoid of $\mathcal M_n(R)$. In consequence, $H$ is Dedekind-finite (as is the case with any cancellative monoid) and, by Remark \ref{rem:preorders}\ref{rem:preorders(1)}, the $\mid_H$-units are exactly the units of $\mathcal M_n(R)$, i.e., the elements of the general linear group ${\rm GL}_n(R)$. \\

\indent{}We claim that $H$ is a weakly l.f.g.u.~monoid. For, let $A$ be a non-unit of $H$. Since $R$ is a commutative PID and hence a UFD, there are an integer $k \ge 1$ and prime elements $p_1, \ldots, p_k \in R$ such that $\det A = \allowbreak p_1 \cdots p_k$. 
Now let $B$ be a divisor of $A$ in $H$. By \cite[Theorem 3.8]{Jac85} (namely, by the existence of a Smith normal form for matrices over a commutative PID), we can find $U, V \in H^\times$ and a diagonal matrix $D := \allowbreak \diag(b_1, \ldots, b_n) \in H$ such that $B = UDV$. It follows (by the Cauchy-Binet formula) that $\det D = b_1\cdots b_n$ divides $\det(A) = p_1 \cdots p_k$ in $R$ and, by the elementary properties of UFDs, there are a sequence $I_1, \dots, I_n$ of pairwise disjoint subsets of $\llb 1, k \rrb$ and units $u_1, \ldots, u_n \in R$ such that $b_i = u_i \wp_i$ for every $i \in \llb 1, n \rrb$, where $\wp_i := \prod_{j \in I_i} p_j$. We can thus write $B = U'D'V$, where $U' := U \diag(u_1, \ldots, u_n) \in H^\times$ and $D' := \allowbreak \diag(\wp_1, \ldots, \wp_n) \in H$. Since there are finitely many choices for the $n$-tuple $(I_1, \dots, I_n)$ and hence for the definition of the matrix $D'$, 
we therefore conclude from the arbitrariness of $B$ that $A$ has finitely many divisors up to associates and then, from the arbitrariness of $A$, that $H$ is a weakly l.f.g.u.~monoid.
\end{enumerate*}

\vskip 0.05cm

\begin{enumerate*}[label=\textup{(\arabic{*})},mode=unboxed, resume]
\item\label{exa:lfgu-premonoids(5)}
We say that a premonoid $(H, \preceq)$ is \evid{finite} if so is the monoid $H$. We have already encountered finite premonoids in Example \ref{exa:properties-of-factorizations}\ref{exa:properties-of-factorizations(1)}. Further examples arising from the interplay between factorization theory and additive combinatorics are given by restricted power monoids (see item \ref{exa:lfgu-premonoids(2)}) of finite abelian groups. Of course, every finite premonoid is loft and f.g., regardless of the actual choice of the preorder.
\end{enumerate*}
\end{examples}

We will see that, under some circumstances, weakly l.f.g.u.~premonoids are factorable (Corollary \ref{cor:wlfgu weakly positive is factorable}), FmF-atomic (Theorem \ref{thm: sufficient for loftness}), or even FF-atomic (Corollary \ref{cor:weakly-lfgu-strongly-pos-is-FF-atom}). We start with a preliminary lemma.

\begin{lemma}\label{lem:removing-the-units}
Let $A$ and $Q$ be subsets of a monoid $H$. If $Q^2 \subseteq Q$, then $\langle QAQ \rangle_H \subseteq \allowbreak Q \cup \allowbreak \langle Q \,(A \setminus Q)\, Q \rangle_H$.
\end{lemma}

\begin{proof}
Let $x \in H$ be an element of the form $u_1 a_1 v_1 \cdots u_n a_n v_n$ (with $n \in \mathbb N^+$), where $u_1, v_1, \ldots, \allowbreak u_n, \allowbreak v_n \in Q$ and $a_1, \ldots, \allowbreak a_n \in A$. We have to prove that $x$ is either in $Q$ or in $\langle Q\,(A \setminus Q)\,Q \rangle_H$.

We may assume that there exists $j \in \allowbreak \llb 1, n \rrb$ with $a_j \in Q$; otherwise, $a_1, \ldots, a_n$ are all in $A \setminus Q$ and there is nothing left to prove because $x \in \langle Q\,(A \setminus Q)\,Q \rangle_H$. If $n = 1$, then $j = 1$ and $x \in Q^3$, so again we are done because $Q^k \subseteq Q$ for all $k \in \mathbb N^+$ (by the hypothesis that $Q^2 \subseteq Q$). If, on the other hand, $n \ge 2$, then we can write $x = \bar{u}_1 b_1 \bar{v}_1 \cdots \bar{u}_{n-1} b_{n-1} \bar{v}_{n-1}$, where we put
	$$
	\bar{u}_i := 
	\left\{
	\begin{array}{ll}
	\! u_i & \text{if } 1 \le i < j \\
	\! u_i a_i v_i u_{i+1} & \text{if } i = j \ne n \\
	\! u_{i+1} & \text{if } j < i \le n-1
	\end{array}
	\right.\!
	\quad\text{and}\quad
	\bar{v}_i := 
	\left\{
	\begin{array}{ll}
	\! v_i & \text{if } 1 \le i < j \le n-1 \\
	\! v_{i+1} & \text{if } j \le i \le n-1 \\
	\! v_i u_{i+1} a_{i+1} v_{i+1} & \text{if } i+1 = j = n
	\end{array}
	\right.\!,
	$$
	and we take $b_i := a_i$ for $1 \le i < j$ and $b_i := a_{i+1}$ for $j \le i \le n-1$.
	It follows, by induction on $n$, that $x \in H$ (as wished), since it is clear from the above that $\bar u_i, \bar v_i \in Q$ and $b_i \in A$ for all $i \in \llb 1, n-1 \rrb$.
\end{proof}

For the next theorem we recall from Definition \ref{def:factorizations}\ref{def:factorizations(4)} that a premonoid $\mathcal H = (H, \preceq)$ is factorable if each $\preceq$-non-unit can be written as a product of $\preceq$-irreducibles; and from Definition \ref{def:preordered-monoids-et-alia} that $\mathcal H$ is a weakly positive monoid if $\mathcal H^\times x \mathcal{H}^\times \preceq x \preceq HxH$ for every $x \in H$.

\begin{theorem}\label{thm:fgu weakly positive}
Let $\mathcal H = (H, \preceq)$ be an f.g.u.~weakly positive monoid. There then exists a finite set $A$ of $\preceq$-irreducibles such that every $\preceq$-non-unit can be written as a product of elements of $\mathscr I(\mathcal H) \cap \mathcal H^\times A \,\mathcal H^\times$. In particular, $\mathcal H$ is factorable.
\end{theorem}

\begin{proof}
Let the set of $\preceq$-non-units be non-empty, or else the conclusion is obvious. We split the proof into two parts: In \textsc{Part 1}, we prove that there is a finite set $A \subseteq \mathscr I(\mathcal H)$ such that $H \setminus \mathcal H^\times \subseteq \langle \mathcal H^\times A\, \mathcal H^\times \rangle_H$; and in \textsc{Part 2}, that $H \setminus \mathcal H^\times \subseteq \langle \mathscr I(\mathcal H) \cap \mathcal H^\times A\, \mathcal H^\times \rangle_H$.

\vskip 0.05cm

\textsc{Part 1:} By the hypothesis that $\mathcal H$ is an f.g.u.~premonoid, there is a subset $A$ of $H$ with $|A| < \infty$ and $H \setminus \mathcal H^\times \subseteq \langle \mathcal H^\times A \, \mathcal H^\times \rangle_H$; and it is, of course, harmless to assume (as we do) that 
\begin{equation}\label{equ:minimal-in-size}
H \setminus \mathcal H^\times \not\subseteq \langle \mathcal H^\times B \, \mathcal H^\times \rangle_H,
\qquad
\text{for every }
B \subseteq H \text{ with } 
|B| < \allowbreak |A|.
\end{equation}
On the other hand, since $\mathcal H$ is a weakly positive monoid, we have from Remark \ref{rem:premonoids}\ref{rem:premonoids(3)} that $(\mathcal H^\times)^2 = \mathcal H^\times$. Therefore, we gather from the above and Lemma \ref{lem:removing-the-units} that 
\[
\emptyset \ne H \setminus \mathcal H^\times \subseteq \langle \mathcal H^\times A \, \mathcal H^\times \rangle_H \setminus \mathcal H^\times \subseteq \langle \mathcal H^\times (A \setminus \mathcal H^\times) \, \mathcal H^\times \rangle_H,
\]
which, in view of Equ.~\eqref{equ:minimal-in-size}, is only possible if $A$ is a set of $\preceq$-non-units. Suppose by way of contradiction that there is an element $a \in A$ that is not $\preceq$-irreducible. Then $a = xy$ for some $\preceq$-non-units $x, y \in H$ with $x \prec \allowbreak a$ and $y \prec a$ (as we have just found that $a$ is a $\preceq$-non-unit); and since every $\preceq$-non-unit lies in the submonoid of $H$ generated by $\mathcal H^\times A \, \mathcal H^\times$, there exist $a_1, \ldots, a_{m+n} \in A$ (with $m, n \in \mathbb N^+$) such that 
\[
x \in \mathcal H^\times a_1 \mathcal H^\times \cdots \mathcal H^\times a_m \mathcal H^\times
\quad\text{and}\quad
y \in \mathcal H^\times a_{m+1} \mathcal H^\times \cdots \mathcal H^\times a_{m+n} \mathcal H^\times.
\]
Thus, using again that $\mathcal H$ is a weakly positive monoid (and hence, in particular, $z \preceq HzH$ for each $z \in \allowbreak H$), we conclude that $a_i \preceq x \prec a$ for every $i \in \llb 1, m \rrb$ and $a_{m+i} \preceq y \prec a$ for every $i \in \llb 1, n \rrb$; that is to say, $a_i \ne a$ for all $i \in \llb 1, m+n \rrb$. This, however, means that $a = xy$ lies in the submonoid of $H$ generated by $\mathcal H^\times \bar{A} \, \mathcal H^\times$, where $\bar{A} := A \setminus \{a\}$. It follows that 
\[
H \setminus \mathcal H^\times \subseteq \langle \mathcal H^\times A \, \mathcal H^\times \rangle_H \subseteq \langle \mathcal H^\times \bar{A} \, \mathcal H^\times \rangle_H,
\]
which is in contradiction with Eq.~\eqref{equ:minimal-in-size} and ultimately shows that $A$ is a set of $\preceq$-irreducibles. 

\vskip 0.05cm

\textsc{Part 2:} It remains to see that every $\preceq$-non-unit factors as a product of elements from the set $I_A := \mathscr I(\mathcal H) \cap \allowbreak \mathcal H^\times A\, \mathcal H^\times$; and since $H \setminus \mathcal H^\times \subseteq \langle \mathcal H^\times A\, \mathcal H^\times \rangle_H$, it will suffice to check that $\mathcal H^\times A\, \mathcal H^\times \subseteq \langle I_A \rangle_H$. 

Assume to the contrary that $\alpha \notin \langle I_A \rangle_H$ for some $\alpha \in \mathcal H^\times A\, \mathcal H^\times$; we will freely use that, by Remark \ref{rem:premonoids}\ref{rem:premonoids(3)}, $H \setminus \mathcal H^\times$ is a two-sided ideal of $H$ and hence $\mathcal H^\times A \mathcal H^\times$ is a set of $\preceq$-non-units. We claim that there exists a se\-quence $\alpha_1, \alpha_2, \ldots$ of elements of $\mathcal H^\times A\, \mathcal H^\times \setminus \langle I_A \rangle_H$ with $\alpha_{i+1} \prec \alpha_i$ for all $i \in \mathbb N^+$. For, put $\alpha_1 := \allowbreak \alpha$ and suppose that, for a certain $k \in \mathbb N^+$, we have recursively found $\alpha_1, \ldots, \alpha_k \in \mathcal H^\times A\, \mathcal H^\times \setminus \langle I_A \rangle_H$ with the property that $\alpha_{i+1} \prec \alpha_i$ for each $i \in \llb 1, k-1 \rrb$. Since $\alpha_k$ is neither a $\preceq$-unit nor a $\preceq$-irreducible (or else $\alpha_k \in \allowbreak I_A \subseteq \allowbreak \langle I_A \rangle_H$), we have $\alpha_k = xy$ for some $\preceq$-non-units $x, y \in H$ with $x \prec \alpha_k$ and $y \prec \alpha_k$. So, there are $u_1, \allowbreak v_1, \ldots, \allowbreak u_{m+n}, v_{m+n} \in \mathcal H^\times$ and $a_1, \ldots, a_{m+n} \in A$ (with $m, n \in \mathbb N^+$) such that 
\[
x = u_1 a_1 v_1 \cdots u_n a_n v_n
\quad\text{and}\quad
y = u_{m+1} a_{m+1} v_{m+1} \cdots u_{m+n} a_{m+n} v_{m+n},
\]
which is only possible if $\alpha_{k+1} := u_i a_i v_i \notin \langle I_A \rangle_H$ for some $i \in \llb 1, m+n \rrb$ (or else $\alpha_k = xy \in \langle I_A \rangle_H$). Since, on the other hand, $\alpha_{k+1} \preceq x \prec \alpha_k$ or $\alpha_{k+1} \preceq y \prec \alpha_k$ (by the hypothesis that $\mathcal H$ is a weakly positive monoid), this is enough to prove our claim (by induction).

Considering that every $\alpha \in \mathcal H^\times A \,\mathcal H^\times$ is $\preceq$-equivalent to an element $a \in A$ (again by the weak positivity of $\mathcal H$), it follows that there is a sequence $a_1, a_2, \ldots$ of elements of $A$ such that $a_{i+1} \preceq \alpha_{i+1} \prec \alpha_i \preceq a_i$ for all $i \in \mathbb N^+$, which is a contradiction because $A$ is a finite set.
\end{proof}

\begin{corollary}\label{cor:wlfgu weakly positive is factorable}
Every weakly l.f.g.u.~weakly positive monoid is factorable.
\end{corollary}
\begin{proof}
Let $\mathcal H = (H, \preceq)$ be a weakly l.f.g.u.~weakly positive monoid and fix a $\preceq$-non-unit $x$; we have to check that $\mathcal{Z}_\mathcal{H}(x)$ is non-empty. Since $\mathcal H$ is a weakly l.f.g.u.~premonoid, the germ $\mathcal K := \llangle x\rrangle_\mathcal{H}$ is f.g.u. On the other hand, we gather from Remark \ref{rem:premonoids}\ref{rem:premonoids(4)} that $\mathcal K$ is a weakly positive monoid. Thus, Theorem \ref{thm:fgu weakly positive} implies that $\mathcal K$ is factorable. It follows that $\mathcal Z_{\mathcal K}(x)$ is non-empty, as we have from Remark \ref{rem:preorders}\ref{rem:preorders(2)} that $x$ is a $\preceq_{\mathcal K}$-non-unit; and by Proposition \ref{prop:x-closedness}\ref{prop:x-closedness(ii)}, this suffices to finish the proof (by definition, the monoid $\llangle x \rrangle_H$ contains all the divisors of $x$ in $H$).
\end{proof}

\begin{remarks}\label{rem:atoms-irreducibles}
\begin{enumerate*}[label=\textup{(\arabic{*})}, mode=unboxed]
    \item\label{rem:atoms-irreducibles(1)} Let $\mathcal{H} = (H,\preceq)$ be a weakly positive monoid and assume the existence of a finite set $A$ of $\preceq$-irreducibles such that $\mathscr I(\mathcal{H}) \subseteq \mathcal{H}^\times A \, \mathcal{H}^\times$; we aim to show that $\mathcal{H}$ is oft. For, let $\mathfrak b = b_1 \ast \cdots \ast b_n$ be a non-empty $\mathscr I(\mathcal{H})$-word of length $n$. Then, for every $i \in \llb 1, n \rrb$, $b_i = u_i a_i v_i$ for some $u_i, v_i \in \mathcal{H}^\times$ and $a_i \in A$; and since $\mathcal{H}$ is weakly positive, $b_i \preceq a_i \preceq b_i$ for each $i \in \llb 1, n \rrb$. Hence $\mathfrak b$ is $\sqeq_\mathcal{H}$-equivalent to the $A$-word $a_1 \ast \dots \ast a_n$, and this suffices to conclude.
\end{enumerate*}

\vskip 0.05cm
\begin{enumerate*}[label=\textup{(\arabic{*})}, mode=unboxed, resume]
    \item\label{rem:atoms-irreducibles(2)} Let $\mathcal H = (H, \preceq)$ be an f.g.u.~weakly positive monoid. By Theorem \ref{thm:fgu weakly positive}, there is a finite $A \subseteq \mathscr I(\mathcal H)$ with the property that $H = \langle \mathcal H^\times A \, \mathcal H^\times \rangle_H$, and we claim $\mathscr{A}(\mathcal{H}) \subseteq \allowbreak \mathcal{H}^\times A \,\mathcal{H}^\times$: By item \ref{rem:atoms-irreducibles(1)}, this will imply that, if $\mathscr{I}(\mathcal{H}) \subseteq \allowbreak \mathscr{A}(\mathcal{H})$, then $\mathcal H$ is oft. For the claim, let $a \in H$ be a $\preceq$-atom. Then $a \ne 1_H$ and hence $a = \allowbreak u_1 a_1 v_1 \cdots \allowbreak u_n a_n v_n$ for some $u_1, v_1, \ldots, u_n, v_n \in \mathcal{H}^\times$ and $a_1, \dots, a_n \in A$ (with $n \in \mathbb N^+$). However, if $n \ge 2$ then we get from Remark \ref{rem:premonoids}\ref{rem:premonoids(3)} that $a$ factors as a product of two $\preceq$-non-units, contradicting that $a$ is a $\preceq$-atom. So, $a = u_1 a_1 v_1 \in \mathcal H^\times A \, \mathcal H^\times$ (as wished).
\end{enumerate*}

\vskip 0.05cm

\begin{enumerate*}[label=\textup{(\arabic{*})}, mode=unboxed, resume]
    \item\label{rem:atoms-irreducibles(3)} As a complement to item \ref{rem:atoms-irreducibles(2)}, we note that not for every l.f.g.u.~positive monoid $\mathcal{H} = (H,\preceq)$ there exists a finite $A \subseteq \mathscr I(\mathcal{H})$ such that $\mathscr I(\mathcal{H})\subseteq \mathcal{H}^\times A \, \mathcal{H}^\times$ (recall from Remark \ref{rem:premonoids}\ref{rem:premonoids(5)} that positive implies weakly positive). E.g., let $H := (\mathbb N, +)$ be the additive monoid of the non-negative integers and, for all $x, y \in \mathbb N$, define $x \preceq y$ if and only if $x = 0$ or $x, y \in \mathbb N^+$. The only $\preceq$-unit is then the identity $0$, the only $\preceq$-atom is $1$, and every non-zero element of $H$ is a $\preceq$-quark and hence a $\preceq$-irreducible. Moreover, the premonoid $\mathcal H = (H, \preceq)$ is obviously positive and f.g.~(in fact, $H$ is generated by $1$). Yet, every positive integer is a $\preceq$-irreducible, with the result that $\mathscr I(\mathcal H)$ is not contained in $\mathcal H^\times A\, \mathcal H^\times = A$ for any \emph{finite} $A \subseteq \allowbreak H$. (Incidentally, $\mathcal H$ is not a strongly positive monoid, because $0 \prec 1$ but $0+1 \preceq 1+1 \preceq 0+1$.)
\end{enumerate*}
\end{remarks}

Given a set $X$, we say that an $X$-word $\mathfrak u$ is a \evid{scattered subword} of an $X$-word $\mathfrak v$  if there is a (strictly) increasing function $\sigma \colon \llb 1, \| \mathfrak u \|_X \rrb \to \llb 1, \| \mathfrak v \|_X \rrb$ such that $\mathfrak u[i] = \mathfrak v[\sigma(i)]$ for each $i \in \llb 1, \| \mathfrak u \|_X \rrb$ (cf.~Example \ref{exa:lfgu-premonoids}\ref{exa:lfgu-premonoids(3)}).
Our interest in scattered subwords lies in the following result, which is commonly referred to as \emph{Higman's lemma} and will be a main ingredient in the proof of Theorem \ref{thm:loft-BmF-is-FmF}\ref{thm:loft-BmF-is-FmF(ii)} and Corollary \ref{cor:weakly-lfgu-strongly-pos-is-FF-atom} (the result will also enter the proof of Theorem \ref{thm:lfgu-left-duo-satisfies-ACCP}, though in a more general form).

\begin{theorem}[Higman's lemma]\label{thm:higman}
	If $X$ is a finite set, then every infinite sequence of $X$-words contains an infinite sub\-se\-quence each of whose terms is a scattered subword of the next.
\end{theorem}
\begin{proof}
This is an immediate corollary of \cite[Theorems 2.1 and 4.3]{Hi52} applied to the finite poset $(X, =_X)$, where $=_X$ is the discrete order on $X$ (meaning that $x =_X y$ if and only if $x = y \in X$).
\end{proof}

As already mentioned in the introduction, Higman's lemma is a non-commutative generalization of Dickson's lemma \cite{Di13}. But while Dickson's lemma has been long known to play a key role in the study of the arithmetic of integral domains and ``nearly cancellative'' commutative monoids (see \cite[Theorem 2.9.13]{GeHK06}, \cite[Proposition 7.3]{Ge16c}, and \cite[Proposition 3.4]{FGKT} for some representative results in this direction), Higman's has never found application in factorization theory until now.

\begin{theorem}\label{thm:loft-BmF-is-FmF}
The following hold for a loft premonoid $\mathcal H = (H, \preceq)$:
	\begin{enumerate}[label=\textup{(\roman{*})}]
	\item\label{thm:loft-BmF-is-FmF(i)} $\mathcal{H}$ is \textup{BF}-factorable (resp., BmF-factorable) if and only if it is \textup{FF}-factorable (resp., \textup{FmF}-factorable).
	\item\label{thm:loft-BmF-is-FmF(ii)} Up to $\sqeq_\mathcal{H}$-equivalence, every $\preceq$-non-unit has finitely many minimal $\preceq$-factorizations.
	\end{enumerate}
\end{theorem}

\begin{proof}
    Fix a $\preceq$-non-unit $x \in H$. The loftness of $\mathcal H$ implies the existence of a finite set $A_x \subseteq \mathscr I(\mathcal H)$ such that every word in the set $\mathcal Z_\mathcal{H}(x)$ of $\preceq$-factorizations of $x$ is $\sqeq_\mathcal{H}$-equivalent to an $A_x$-word.
    
    \vskip 0.05cm
    
	\ref{thm:loft-BmF-is-FmF(i)} We have already observed in Remark \ref{rem:diagram}\ref{rem:diagram(2)} that an FF-factorable premonoid is BF-factorable. So, let $\mathcal H$ be a BF-factorable premonoid. The set of lengths $\mathsf L_\mathcal{H}(x)$ of $x$ is then finite and non-empty. Assume for a contradiction that $\mathcal Z_\mathcal{H}(x)$ contains infinitely many $\sqeq_\mathcal{H}$-inequivalent elements (note that $\mathcal Z_\mathcal{H}(x)$ is non-empty, because $\mathcal H$ is BF-factorable). The same is then true for the set
    \begin{equation}
    \Lambda(x) := \{\bar{\mathfrak a} \in \mathscr F(A_x) \colon \bar{\mathfrak a} \sqeq_\mathcal{H} \mathfrak a \sqeq_\mathcal{H} \bar{\mathfrak a}, \text{ for some }\mathfrak a \in \mathcal Z_\mathcal{H} (x)\}
    \end{equation}
    Since $\Lambda(x)$ is a subset of $\mathscr F(A_x)$ and $A_x$ is finite, it follows that there is a sequence $\bar{\mathfrak a}_1, \bar{\mathfrak a}_2, \ldots$ of non-empty $A_x$-words in $\Lambda(x)$ with $\|\bar{\mathfrak{a}}_i\|_H < \|\bar{\mathfrak{a}}_{i+1}\|_H$ for all $i \in \mathbb N^+$. But this implies that the set $\mathsf L_\mathcal{H}(x)$ is infinite (absurd), for each $A_x$-word $\bar{\mathfrak a} \in \Lambda(x)$ is $\sqeq_\mathcal{H}$-equivalent to a $\preceq$-factorization $\mathfrak a$ of $x$ and hence $\|\bar{\mathfrak a}\|_H = \|\mathfrak a\|_H$. 
    
    An analogous argument applied to minimal $\preceq$-factorizations of $x$ shows that $\mathcal{H}$ is BmF-factorable only if it is FmF-factorable, and the converse follows from Remark \ref{rem:diagram}\ref{rem:diagram(2)}.
	
	\vskip 0.1cm
	\ref{thm:loft-BmF-is-FmF(ii)} Suppose to the contrary that there is a  $\preceq$-non-unit $x$ such that $\mathcal Z_\mathcal{H}^{\sf m}(x)$ contains infinitely many $\sqeq_\mathcal{H}$-inequivalent elements. Then, also the set
    \begin{equation}\label{equ:Lambda}
    \Lambda^{\rm m}(x) := \{\bar{\mathfrak a} \in \mathscr F(A_x) \colon \bar{\mathfrak a} \sqeq_\mathcal{H} \mathfrak a \sqeq_\mathcal{H} \bar{\mathfrak a}, \text{ for some }\mathfrak a \in \mathcal Z^{\rm m}_\mathcal{H} (x)\}
    \end{equation}
    is infinite. So, arguing as in the proof of item \ref{thm:loft-BmF-is-FmF(i)}, we get a sequence $\bar{\mathfrak a}_1, \bar{\mathfrak a}_2, \ldots$ of non-empty $A_x$-words in $\Lambda^{\rm m}(x)$ with $\|\bar{\mathfrak{a}}_i\|_H < \|\bar{\mathfrak{a}}_{i+1}\|_H$ for every $i \in \mathbb N^+$; and by Higman's lemma, there is no loss of generality in assuming that $\bar{\mathfrak{a}}_i$ is a scat\-tered sub\-word of $\bar{\mathfrak{a}}_{i+1}$. But then, by the very definition of the pre\-order $\sqeq_\mathcal{H}$, we get $\bar{\mathfrak a}_1 \sqneq_\mathcal{H} \bar{\mathfrak a}_2$ and conclude from Eq.~\eqref{equ:Lambda} that $\mathfrak a_1 \sqeq_\mathcal{H} \bar{\mathfrak a}_1 \sqneq_\mathcal{H} \bar{\mathfrak a}_2 \sqeq_\mathcal{H} \mathfrak a_2$ for some $\mathfrak a_1, \mathfrak a_2 \in \mathcal Z^{\rm m}_\mathcal{H} (x)$, which is impossible since a minimal $\preceq$-factorization of $x$ is a $\sqeq_\mathcal{H}$-minimal word in the set $\pi_H^{-1}(x) \cap \mathscr F(\mathscr{I}(\mathcal{H}))$.
\end{proof}

We note in passing that the use of Higman's lemma in the proof of Theorem \ref{thm:loft-BmF-is-FmF}\ref{thm:loft-BmF-is-FmF(ii)} is not really nec\-es\-sar\-y: Dickson's lemma would work just fine, at the cost of making the proof slightly longer.

\begin{theorem}\label{thm: sufficient for loftness}
	Every weakly l.f.g.u.~weakly positive monoid $\mathcal H = (H, \preceq)$ such that the $\preceq$-irreducibles are $\preceq$-atoms, is loft and hence FmF-atomic.
\end{theorem}

\begin{proof}
By Corollary \ref{cor:wlfgu weakly positive is factorable} and Theorem \ref{thm:loft-BmF-is-FmF}\ref{thm:loft-BmF-is-FmF(ii)}, it suffices to show that $\mathcal{H}$ is loft. Let $x$ be a $\preceq$-non-unit and consider the germ $\llangle x\rrangle_\mathcal{H}$ of $\mathcal H$ at $x$. To ease the notation, we put $K:= \llangle x\rrangle_H$ and $\mathcal{K}:=\llangle x\rrangle_\mathcal{H}$. Since $\mathcal{H}$ is a weakly l.f.g.u.~premonoid, $\mathcal{K}$ is f.g.u.; and since $\mathcal H$ is weakly positive, we get from Remark \ref{rem:premonoids}\ref{rem:premonoids(4)} that $\mathcal K$ is itself weakly positive. It then follows from Theorem \ref{thm:fgu weakly positive} that there exists a finite $A \subseteq \mathscr{I}(\mathcal{K})$ such that $K\setminus \mathcal K^\times \subseteq \allowbreak \langle \mathcal{K}^\times A\, \mathcal{K}^\times\rangle_K=\langle \mathcal{K}^\times A\, \mathcal{K}^\times\rangle_H$. We claim that every $\preceq$-factorization of $x$ is $\sqeq_\mathcal{H}$-equivalent to an $A_x$-word, where $A_x := A\cap \mathscr{I}_x(\mathcal{H})$ is a finite subset of $\mathscr I(\mathcal H)$. 

For, let $\mathfrak a = a_1\ast \dots \ast a_n$ be a $\preceq$-factorization of $x$ (note that $\mathfrak a$ cannot be the empty word). Since $x = \allowbreak a_1\cdots a_n$ and, by hypothesis, $\mathscr{I}(\mathcal{H}) = \mathscr{A}(\mathcal{H})$, we have from Proposition \ref{prop:x-closedness}\ref{prop:x-closedness(i)} that, for every $i \in \llb 1,n\rrb$, $a_i \in \mathscr{I}_x (\mathcal{H}) = \mathscr{A}_x(\mathcal{H}) = \allowbreak \mathscr{A}_x(\mathcal{K}) \subseteq \mathscr{A}(\mathcal{K})$. Thus, since $\mathscr{A}(\mathcal{K})\subseteq \mathcal{K}^\times A \,\mathcal{K}^\times$ by Remark \ref{rem:atoms-irreducibles}\ref{rem:atoms-irreducibles(2)},  $a_i=u_ib_iv_i$ for some $u_i, v_i \in \allowbreak \mathcal{K}^\times$ and $b_i\in A$, which in turn shows that $b_i \in \allowbreak A \cap \mathscr{I}_x(\mathcal{K}) = \allowbreak A \cap \mathscr{I}_x(\mathcal{H})$. So, arguing as in Remark \ref{rem:atoms-irreducibles}\ref{rem:atoms-irreducibles(1)}, we conclude (as wished) that $a_1 \ast \dots \ast a_n$ is $\sqeq_\mathcal{H}$-equivalent to $b_1 \ast \allowbreak \dots \ast \allowbreak b_n$ (recall from Remark \ref{rem:preorders}\ref{rem:preorders(2)} that $\mathcal{K}^\times \subseteq\mathcal{H}^\times$).
\end{proof}

\begin{corollary}
\label{cor:weakly-lfgu-strongly-pos-is-FF-atom}
Every weakly l.f.g.u.~strongly positive monoid is FF-atomic.
\end{corollary}

\begin{proof}
Let $\mathcal H = (H, \preceq)$ be a weakly l.f.g.u.~strongly positive monoid and write $\approx$ for the relation of $\preceq$-equivalence. By Remark \ref{rem:premonoids}\ref{rem:premonoids(2)}, every $\preceq$-irreducible is a $\preceq$-atom; and by Theorem \ref{thm: sufficient for loftness}, $\mathcal H$ is then loft and FmF-atomic (recall from Remark \ref{rem:premonoids}\ref{rem:premonoids(5)} that every strongly positive monoid is weakly positive). 

Suppose for a contradiction that $\mathcal H$ is not FF-atomic. 
Then, by Theorem \ref{thm:loft-BmF-is-FmF}\ref{thm:loft-BmF-is-FmF(i)}, $\mathcal H$ is not BF-atomic. So, there is a sequence $\mathfrak a_1, \mathfrak a_2, \ldots$ of non-empty atomic $\preceq$-factorizations of a certain $\preceq$-non-unit $x \in H$ with $\|\mathfrak a_i\|_H < \|\mathfrak a_{i+1}\|_H$ for each $i \in \mathbb N^+$. On the other hand, the same argument used to prove loftness in Theorem \ref{thm: sufficient for loftness} shows that there is a finite set $A_x \subseteq \allowbreak \mathscr I_x(\mathcal H) = \allowbreak \mathscr A_x(\mathcal H)$ such that every $\preceq$-factorization $\mathfrak a$ of $x$ is $\sqeq_\mathcal{H}$-equivalent to an $A_x$-word $\mathfrak b$ with $\|\mathfrak a\|_H = \|\mathfrak b\|_H$ and $\mathfrak a[j]\approx \mathfrak b[j]$ for all $j \in \allowbreak \llb 1, \|\mathfrak a\|_H \rrb$. Therefore, for each $i \in \mathbb N^+$, there exists an $A_x$-word $\mathfrak b_i$ of the same length as $\mathfrak a_i$ such that $\mathfrak a_i[j] \approx \allowbreak \mathfrak b_i[j]$ for all $j \in \llb 1, \allowbreak \|\mathfrak a_i\|_H \rrb$, which shows that $x = \pi_H(\mathfrak a_i) \approx \pi_H(\mathfrak b_i)$ by Remark \ref{rem:premonoids}\ref{rem:premonoids(2)} and the fact that $\mathcal H$ is a preordered monoid. But the finiteness of $A_x$ implies, by Higman's lemma, that ${\mathfrak b}_i$ is a scattered subword of $\mathfrak b_j$ for some $i, j \in \mathbb N^+$ with $i < j$ and hence $\|{\mathfrak b}_i\|_H < \|{\mathfrak b}_j\|_H$. Since $\mathcal H$ is strongly positive, this however means by Remark \ref{rem:premonoids}\ref{rem:premonoids(2)} that $x \approx \allowbreak \pi_H({\mathfrak b}_i) \prec \allowbreak \pi_H({\mathfrak b}_j) \approx \allowbreak x$, which is impossible.
\end{proof}

\begin{remark}\label{rem:4.14}
As a complement to Corollary \ref{cor:weakly-lfgu-strongly-pos-is-FF-atom}, note that the premonoid $\mathcal H = (H, \preceq)$ in Remark \ref{rem:atoms-irreducibles}\ref{rem:atoms-irreducibles(3)} (where $H$ is the additive monoid of the non-negative integers) is positive, weakly l.f.g.u, and UF-atomic (and hence FF-atomic), in spite of it not being strongly positive: Every positive integer $n$ can be uniquely written as the sum of $n$ ones and $1$ is the only $\preceq$-atom of $H$. On the other hand, it is clear that $\mathcal H$ is FF-factorable but not UF-factorable, since there are finitely many ways to write a positive integer as a sum of positive integers and every positive integer is, in fact, a $\preceq$-irreducible.
\end{remark}

Strongly positive monoids abound in nature: Some of them have already been discussed in items \ref{exa:preord-mons(2)} and \ref{exa:preord-mons(3)} of Examples \ref{exa:preord-mons}, and a few more can be found in \cite{Tr15}. The range of application of Corollary \ref{cor:weakly-lfgu-strongly-pos-is-FF-atom} is therefore wide.

\section{Back to the classical theory}
\label{sec:duo-mons}
Below, we discuss some applications of the main results from the previous sections to premonoids of the form $(H, \mid_H)$ in which $H$ is a Dedekind-finite monoid and hence to the classical theory of factorizations. For, it is perhaps worth recalling from Definition \ref{def:factorizations}\ref{def:factorizations(6)} that we denote by $\mathscr{I}(H)$ the set of ir\-red\-u\-ci\-bles (that is, $\mid_H$-ir\-red\-u\-ci\-bles) and by $\mathscr{A}(H)$ the set of $\mid_H$-atoms; and from Remark \ref{rem:preorders}\ref{rem:preorders(1)} that $\mathscr{A}(H)$ coincides with the set of (ordinary) atoms of $H$ whenever $H$ is Dedekind-finite.

\begin{theorem}\label{thm:Dedekind-finite wlfgu FmF-atomic}
Every Dedekind-finite weakly l.f.g.u.~monoid $H$ is factorable. If, in addition, every ir\-red\-u\-ci\-ble of $H$ is an (ordinary) atom, then $H$ is FmF-atomic.
\end{theorem}

\begin{proof}
By Definition \ref{def:lfgu}\ref{def:lfgu(4)}, $H$ is a weakly l.f.g.u.~monoid if and only if $(H, \mid_H)$ is a weakly l.f.g.u.~premonoid. On the other hand, we have from Example \ref{exa:preord-mons}\ref{exa:preord-mons(1)} and the Dedekind-finiteness of $H$ that $(H, \mid_H)$ is a weakly positive monoid. So, $H$ is a factorable monoid by Corollary \ref{cor:wlfgu weakly positive is factorable}. If, in addition, $\mathscr{I}(H) = \mathscr{A}(H)$, then the hypotheses of Theorem \ref{thm: sufficient for loftness} are satisfied and hence $H$ is FmF-atomic.
\end{proof}

\begin{corollary}\label{cor:acyclic-is-FmF-atomic}
Every acyclic weakly l.f.g.u.~monoid is FmF-atomic.
\end{corollary}

\begin{proof}
Every acyclic monoid is unit-cancellative and hence, by \cite[Proposition~2.30]{Fa-Tr18}, Dedekind-finite. Moreover, we noted in Example \ref{exa:irrds-atoms-quarks}\ref{exa:irrds-atoms-quarks(1)} that the irreducibles of an acyclic monoid are no different than the (ordinary) atoms. So, the statement follows from Theorem \ref{thm:Dedekind-finite wlfgu FmF-atomic}. 
\end{proof} 

In particular, Corollary \ref{cor:acyclic-is-FmF-atomic} leads to the following generalization of \cite[Proposition 2.7.8.4]{GeHK04} and \cite[Proposition 3.4]{FGKT}, where it is \emph{essentially} proved (though in a different terminology) that every commutative, unit-cancellative, l.f.g.u.~monoid is FF-atomic.

\begin{corollary}\label{cor:comm-unit.canc-weakly.lfgu-is-FFatom}
Every unit-cancellative, weakly l.f.g.u., commutative monoid is FF-atomic.
\end{corollary}

\begin{proof}
It is obvious that a commutative monoid is acyclic if and only if it is unit-cancellative (cf.~Example \ref{exa:acyclic-monoids}\ref{exa:acyclic-monoids(4)}). On the other hand, if $H$ is a unit-cancellative commutative monoid, then we have from Example \ref{exa:irrds-atoms-quarks}\ref{exa:irrds-atoms-quarks(1)} that an irreducible of $H$ is an (ordinary) atom and hence from \cite[Proposition 4.7(v)]{An-Tr18} that every $\mid_H$-factorization is actually a minimal atomic $\mid_H$-factorization. Stitching the pieces together, we thus see from Corollary \ref{cor:acyclic-is-FmF-atomic} that a unit-cancellative, weakly l.f.g.u., commutative monoid is FF-atomic. 
\end{proof}

Further examples of acyclic monoids to which one can apply Corollary \ref{cor:acyclic-is-FmF-atomic} are listed below:

\begin{examples}\label{exa:acyclic-monoids}
\begin{enumerate*}[label=\textup{(\arabic{*})}, mode=unboxed]
\item\label{exa:acyclic-monoids(1)} Let $f \colon H \to K$ be a monoid homomorphism with $f^{-1}(K^\times)\subseteq H^\times$ and $K$ acyclic. If $x = uxv$ for some $u, v, x \in H$, then $f(u), f(v) \in K^\times$ and hence $u, v \in H^\times$, i.e., $H$ is itself acyclic. Suppose, on the other hand, that $K$ is also commutative (and hence unit-cancellative). \\ 
    
\indent{}Given a non-unit $x \in H$, we claim that every word in the set $\pi_H^{-1}(x) \cap \mathscr F(H \setminus H^\times)$ is $\sqeq_\mathcal{H}$-minimal, where $\mathcal{H} := (H, \mid_H)$. For, assume to the contrary that $\mathfrak b \sqneq_\mathcal{H} \mathfrak a$ for some words $\mathfrak a, \mathfrak b \in \mathscr F(H \setminus H^\times)$ such that $\pi_H(\mathfrak a) = \pi_H(\mathfrak b)=x$ and, to ease notation, put $n := \|\mathfrak a\|_H$ and $k := \|\mathfrak b\|_H$. Then $1 \le k < n$ and there is an injective function $\sigma \colon \llb  1, k \rrb \to \llb 1, n \rrb$ such that $\mathfrak b[i]$ is $\mid_H$-equivalent to  $\mathfrak a[\sigma(i)]$ for every $i \in \llb 1, k \rrb$. Since $H$ is acyclic (from the above), this means that $\mathfrak b[i] \in H^\times \mathfrak a[\sigma(i)] H^\times$ for each $i \in \llb 1, k \rrb$; and since $K$ is commutative and monoid homomorphisms send units to units, there is then $u \in H^\times$ such that
\[
f(\mathfrak a[1]) \cdots f(\mathfrak a[n]) = f(\pi_H(\mathfrak a)) = f(x) = f(\pi_H(\mathfrak b)) = f(\mathfrak b[1]) \cdots f(\mathfrak b[k])=u f(\mathfrak a[\sigma(1)]) \cdots f(\mathfrak a[\sigma (k)]). 
\]
By unit-cancellativity of $K$, we thus find that $f(\mathfrak a[j]) \in \allowbreak K^\times$ for some $j \in \llb 1, n \rrb \setminus \{\sigma(1), \ldots, \sigma(k)\}$ (recall that $k < n$), which is impossible because $\mathfrak a[j]$ is a non-unit of $H$ and $f$ maps non-units to non-units.
\end{enumerate*}

\vskip 0.05cm

\begin{enumerate*}[label=\textup{(\arabic{*})}, mode=unboxed, resume]
\item\label{exa:acyclic-monoids(2)} In the notation of Example \ref{exa:lfgu-premonoids}\ref{exa:lfgu-premonoids(4)}, the function $f \colon \mathcal M_n(R) \to R \colon A \mapsto \det A$ yields a monoid homomorphism from $H$ to the monoid $K$ of non-zero elements of $R$. Since $f^{-1}(K^\times) \subseteq H^\times$ and $K$ is commutative and cancellative, it then follows from Corollary \ref{cor:acyclic-is-FmF-atomic} and item \ref{exa:acyclic-monoids(1)} that $H$ is FF-atomic (recall the diagram in Remark \ref{rem:diagram}\ref{rem:diagram(2)} and cf.~the block ``Matrix rings'' on p.~531 of \cite{Ba-Sm15}, where Baeth and Smertnig show that $f$ is a transfer homomorphism as per \cite[Definition 2.1(1)]{Ba-Sm15} and hence $H$ is HF-atomic).
\end{enumerate*}

\vskip 0.05cm

\begin{enumerate*}[label=\textup{(\arabic{*})}, mode=unboxed, resume]
\item\label{exa:acyclic-monoids(3)} Let $\mathcal H = (H, \preceq)$ be a strongly positive monoid with $\mathcal{H}^\times=\{1_H\}$, and pick $u, v, x \in H$. If $1_H \prec u$ or $1_H \prec v$, then we get from Remark \ref{rem:premonoids}\ref{rem:premonoids(2)} that $x \prec uxv$, which is enough to conclude that $H$ is acyclic when considering that $1_H \preceq H$ (and the only $\preceq$-unit of $H$ is the identity).
\end{enumerate*}

\vskip 0.05cm

\begin{enumerate*}[label=\textup{(\arabic{*})}, mode=unboxed, resume]
\item\label{exa:acyclic-monoids(4)} We have already observed in Example \ref{exa:irrds-atoms-quarks}\ref{exa:irrds-atoms-quarks(1)} that, in general, an acyclic monoid is unit-cancellative but not the other way around. Assume, however, that $H$ is a unit-cancellative duo monoid (see Example \ref{exa:preord-mons}\ref{exa:preord-mons(2)} for the terminology). If $x = uxv$ for some $u, v, x \in H$, then $x = \allowbreak xu'v = uv'x$ for certain $u', v' \in H$ (since $H$ is duo) and hence $u, v \in H^\times$ (because $H$ unit-cancellative and unit-cancellative monoids are Dedekind-finite): This shows that $H$ is acyclic. 
\end{enumerate*}

\vskip 0.05cm

\begin{enumerate*}[label=\textup{(\arabic{*})}, mode=unboxed, resume]
\item\label{exa:acyclic-monoids(5)}
We have from \cite[Corollary 4.6]{Tr20(c)} that a unit-cancellative monoid $H$ satisfying both the ACC on principal left ideals and the ACC on principal right ideals is acyclic and satisfies the ACC on principal two-sided ideals (cf.~Remark \ref{rem:ACCP} and references therein).
\end{enumerate*}

\vskip 0.05cm

\begin{enumerate*}[label=\textup{(\arabic{*})}, mode=unboxed, resume]
\item\label{exa:acyclic-monoids(6)}
A Dedekind-finite BF-factorable monoid $H$ is acyclic. For, assume to the contrary that there exist $u, v, x \in H$ with $u \notin H^\times$ (resp., $v \notin H^\times$) such that $uxv = \allowbreak x$. Then a routine induction shows that $u^n x v^n = \allowbreak x$ for all $n \in \mathbb N$; and since $H$ is Dedekind-finite, we find that $x$ is a non-unit (or else $uxv \in \allowbreak H^\times$ and therefore $u, v \in H^\times$) and hence so are $u^n$ and $x v^n$ (resp., $u^n x$ and $v^n$). Considering that $H$ is factorable (and each non-unit is thus a product of irreducibles), it follows that $x = u^n xv^n$ has a fac\-tor\-i\-za\-tion into $n+1$ or more irreducibles for every $n \in \mathbb N^+$. To wit, $H$ is not BF-factorable (absurd).
\end{enumerate*}
\end{examples}

\begin{remark} 
The condition that the irreducibles of a Dedekind-finite weakly l.f.g.u.~monoid $H$ are (ordinary) atoms, is sufficient but not necessary for $H$ to be FmF-factorable. 

In fact, we know from Example \ref{exa:lfgu-premonoids}\ref{exa:lfgu-premonoids(2)} that the reduced power monoid $\mathcal P_{\mathrm{fin},1}(M)$ of a monoid $M$ is loft and weakly l.f.g.u., and becomes a weakly positive monoid under the divisibility preorder. It thus follows from Theorems \ref{thm:fgu weakly positive} and  \ref{thm:loft-BmF-is-FmF}\ref{thm:loft-BmF-is-FmF(ii)} applied to $(\mathcal{P}_{{\rm fin},1}(M),\mid_{\mathcal{P}_{{\rm fin},1}(M)})$ that $\mathcal{P}_{{\rm fin},1}(M)$ is an FmF-fac\-tor\-able monoid (which, by the way, complements the conclusions of \cite[Theorem 4.13]{An-Tr18} and \cite[Sect.~4.2]{Tr19a} on the atomicity of reduced power monoids). Yet, we noted in the same Example \ref{exa:lfgu-premonoids}\ref{exa:lfgu-premonoids(2)} that $\mathcal{P}_{{\rm fin},1}(M)$ is a Dedekind-finite, weakly l.f.g.u.~monoid whose irreducibles are not in general atoms.
\end{remark}

Next, we construct an FmF-atomic, f.g., reduced monoid $H$ that (i) has the property that every ir\-red\-u\-ci\-ble is an atom, and (ii) does not satisfy the ACCP (Remark \ref{rem:ACCP}) and hence is not BF-atomic (to the contrary of what happens in the classical case with cancellative, f.g., commutative monoids \cite[Proposition 2.7.8.4]{GeHK06}).

\begin{example}
\label{exa:non-atomic-2-generator-1-relator-canc-mon}
Given a set $X$ and a (binary) relation $R$ on the free monoid $\mathscr F(X)$, we define $R^\sharp$ as the smallest monoid congruence on $\mathscr F(X)$ containing $R$.
This means that $\mathfrak u \equiv \mathfrak v \bmod R^\sharp$, for some $X$-words $\mathfrak u$ and $\mathfrak v$, if and only if there are $\mathfrak z_0,\, \mathfrak z_1, \, \ldots,\, \mathfrak z_n \in \mathscr F(X)$ with $\mathfrak z_0 = \mathfrak u$ and $\mathfrak z_n = \mathfrak v$ such that, for each $i \in \allowbreak \llb 0, n-1 \rrb$, there exist $X$-words $\mathfrak p_i$, $\mathfrak q_i$, $\mathfrak q_i^\prime$, and $\mathfrak r_i$ with the following properties:
\begin{center}
(i) either $\mathfrak q_i = \mathfrak q_i^\prime$, or $\mathfrak q_i \RR \mathfrak q_i^\prime$, or $\mathfrak q_i^\prime \RR \mathfrak q_i$; \hskip 1cm (ii) $\mathfrak z_i = \mathfrak p_i \ast \mathfrak q_i \ast \mathfrak r_i$ and $\mathfrak z_{i+1} = \mathfrak p_i \ast \mathfrak q_i^\prime \ast \mathfrak r_i$.
\end{center}
We denote by $\mathrm{Mon}\langle X \mid R \rangle$ the monoid obtained by taking the quotient of $\mathscr F(X)$ by the congruence $R^\sharp$,
write it multiplicatively, and call it a (\evid{monoid}) \evid{presentation}. 

Now, fix an integer $n \ge 2$ and let $H$ be the monoid defined by the presentation $\mathrm{Mon}\langle X \mid R \rangle$, where $X$ is the $2$-element set $\{x, y\}$ and $R := \{(x^{\ast n}, y * x^{\ast n} * y)\} \subseteq \mathscr F(X) \times \mathscr F(X)$. By \cite[Example 4.8]{Tr20(c)}, $H$ is an atomic, reduced, cancellative monoid that does not satisfy the ACCP and each of whose irreducibles is an atom. On the other hand, $H$ is f.g. So, we gather from Theorem \ref{thm:Dedekind-finite wlfgu FmF-atomic} that $H$ is FmF-atomic. Yet, $H$ is not BF-atomic, or else it would satisfy the ACCP by \cite[Theorem 2.28(iii) and Corollary 2.29]{Fa-Tr18}.
(Incidentally, the same presentation (with $n = 1$ or $2$) was considered by A.~Geroldinger in \cite[p.~969]{Ge16c}.) 
\end{example}

Example \ref{exa:non-atomic-2-generator-1-relator-canc-mon} shows not only that Theorem \ref{thm:Dedekind-finite wlfgu FmF-atomic} is, in a sense, sharp; but also that the results obtained in this paper make it possible to draw (non-trivial) arithmetic conclusions on the \emph{existence} of certain fac\-tor\-i\-za\-tions in cases where there is no obvious way of resorting to Theorem \ref{thm:abstract-factorization}. As we are about to see, things are different if we restrict ourselves to left (or right) duo monoids (Example \ref{exa:preord-mons}\ref{exa:preord-mons(2)}). 

\begin{lemma}\label{lem:pseudo-commutativity-in-duo-monoids}
Given a left duo monoid $H$, we have $
Hx_1H \cdots Hx_n H \subseteq \allowbreak Hx_{\sigma(1)} \cdots \allowbreak x_{\sigma(m)}$
for all $x_1, \allowbreak \ldots, \allowbreak x_n \in H$ and every \textup{(}strictly\textup{)} increasing function $\sigma \colon \llb 1, m \rrb \to \llb 1, n \rrb$.
\end{lemma}

\begin{proof}
We proceed by induction on $m$. The base case $m = 1$ is trivial: For all $x_1, \ldots, x_n \in H$ and each $i \in \llb 1, n \rrb$, we have $x_iH \subseteq Hx_i$ (by the fact that $H$ is a left duo monoid) and hence 
\[
Hx_1H \cdots Hx_nH \subseteq Hx_i H \subseteq H^2 x_i = Hx_i.
\]
Now pick $\mu \in \mathbb N^+$, assume inductively that the claim holds for every integer $m$ between $1$ and $\mu$, and fix $x_1, \ldots, x_n \in H$ and an increasing function $\sigma \colon \llb 1, \mu+1 \rrb \to \llb 1, n \rrb$.
It is clear that 
$\sigma(\mu)$ is a positive integer (strictly) smaller than $n$. We can therefore consider the (well-defined) increasing func\-tion $
\llb 1, \mu \rrb \to \llb 1, \sigma(\mu) \rrb \colon i \mapsto \sigma(i)$ and derive from the inductive hypothesis that
\begin{equation*}
Hx_1H \cdots Hx_nH 
= (Hx_1H \cdots Hx_{\sigma(\mu)}H) (Hx_{\sigma(\mu)+1}H \cdots Hx_n) 
\subseteq (H x_{\sigma(1)} \cdots x_{\sigma(\mu)}) Hx_{\sigma(\mu+1)}.
\end{equation*}
So, using that $x_{\sigma(1)} \cdots x_{\sigma(\mu)} H \subseteq Hx_{\sigma(1)} \cdots x_{\sigma(\mu)}$, we conclude that
\[
Hx_1H \cdots Hx_nH \subseteq H^2 x_{\sigma(1)} \cdots x_{\sigma(\mu)} x_{\sigma(\mu+1)}  = H x_{\sigma(1)} \cdots x_{\sigma(\mu+1)},
\]
which, by induction on $m$, suffices to finish the proof (since $\mu$ was an arbitrary positive integer).
\end{proof}

For the next theorem we say after \cite[Sect.~2, p.~328]{Hi52} that a preset $(X, \preceq)$ satisfies the \evid{Erd\H{o}s-Rado condition} if, for every sequence $x_1, x_2, \ldots$ of elements of $X$, there exist $i, j \in \mathbb N^+$ with $i < j$ and $x_i \preceq x_j$ (we use ``preset'' as a shortening of ``preordered set''). 
We gather from \cite[Theorems 2.1 and 4.3]{Hi52} that, if $(X, \preceq)$ satisfies the Erd\H{o}s-Rado condition, then so does the preset $(\mathscr F(X), \preceq_X)$, where $\preceq_X$ is the preorder on (the carrier set of) the free monoid $\mathscr F(X)$ defined by $\mathfrak u \preceq_X \mathfrak v$ if and only if $\mathfrak u$ and $\mathfrak v$ are $X$-words for which there is a (strictly) increasing function $\sigma \colon \llb 1, \|\mathfrak u\|_X \rrb \to \llb 1, \|\mathfrak v\|_X \rrb$ such that $\mathfrak u[i] \preceq \allowbreak \mathfrak v[\sigma(i)]$ for each $i \in \llb 1, \|\mathfrak u\|_X \rrb$. This is a gen\-er\-al\-i\-za\-tion of Higman's lemma (i.e., Theorem \ref{thm:higman}), herein referred to as \emph{Higman's full lemma}.

\begin{theorem}\label{thm:lfgu-left-duo-satisfies-ACCP}
Every left duo, l.f.g.u.~monoid  satisfies the \textup{ACCP}.
\end{theorem}

\begin{proof}
Let $H$ be a left duo, l.f.g.u.~monoid. By Remark \ref{rem:ACCP}, $H$ satisfies the \textup{ACCP} if and only if so does $\llb x \rrb_H$ for every $x \in H$. Moreover, every divisor-closed submonoid $K$ of $H$ is itself left duo: For all $a, b \in K$, we have $ab = ca$ for some $c \in H$ (because $aK \subseteq aH \subseteq Ha$), which yields $c \in K$ and hence $aK \subseteq Ka$ (because $K$ is divisor-closed in $H$ and we have $c \mid_H ab \in K$). So, there is no loss of generality in assuming (as we do) that $H$ is f.g.u.

Accordingly, let $a_1, \ldots, a_n$ be an enumeration of a non-empty finite $A \subseteq H$ such that $H = \langle A^\prime \rangle_H$, where $A^\prime := H^\times A\, H^\times$; and suppose for a contradiction that $H$ does not have the \textup{ACCP}, viz., there exists an infinite sequence $x_1, x_2, \ldots$ of non-units of $H$ that is (strictly) decreasing with respect to the divisibility preorder $\mid_H$. 
For each $k \in \mathbb N^+,$ there is then a non-empty $A'$-word $\mathfrak a_k$ 
such that $x_k = \pi_H(\mathfrak a_k)$, where $\pi_H$ is the factorization homomorphism of $H$;
and the finiteness of $A$ guarantees that the preset $(A', \preceq)$ satisfies the Erd\H{o}s-Rado condition, where $\preceq$ is the preorder on $A'$ defined by $a \preceq b$ if and only if $a \in A'$ and $b \in \allowbreak H^\times a H^\times$. By Higman's full lemma applied to $(A', \preceq)$, we thus find that there are $h, k \in \mathbb N^+$ with $h < k$ and a (strictly) increasing function $\sigma \colon \llb 1, m \rrb \to \llb 1, n \rrb$ such that $\mathfrak a_k[\sigma(i)] \in H^\times \mathfrak a_h[i] H^\times \subseteq H \mathfrak a_h[i] H$ for each $i \in \llb 1, m \rrb$, where $m := \|\mathfrak a_h\|_H$ and $n := \|\mathfrak a_k\|_H$. So, we get from Lemma \ref{lem:pseudo-commutativity-in-duo-monoids} (applied twice) that
\[
\begin{split}
\mathfrak a_k[1] \cdots \mathfrak a_k[n] & \in H \mathfrak a_k[\sigma(1)] \cdots \mathfrak a_k[\sigma(m)] \subseteq H \mathfrak a_h[1] H \cdots H \mathfrak a_h[m] H \subseteq H \mathfrak a_h[1] \cdots \mathfrak a_h[m],
\end{split}
\]
which means that $x_k = \pi_H(\mathfrak a_k) \in H \pi_H(\mathfrak a_h) = x_h$ and hence 
$x_h \mid_H x_k$. This is however impossible, as we have assumed that $x_1, x_2, \ldots$ is a $\mid_H$-decreasing sequence.
\end{proof}

\begin{corollary}\label{cor:left-duo-l.f.g.u.-is-factorable}
In a left duo, l.f.g.u.~monoid, every non-unit factors as a product of irreducibles.
\end{corollary}

\begin{proof}
By Theorem \ref{thm:abstract-factorization} and Remarks \ref{rem:preorders}\ref{rem:preorders(1)} and \ref{rem:ACCP}, every non-unit of a Dedekind-finite monoid satisfying the \textup{ACCP} factors as a product of irreducibles. On the other hand, it is easily checked that every left duo monoid $H$ is Dedekind-finite (if $1_H = uv$ for some $u, v \in H$, then $1_H \in Hu$ and hence $u \in H^\times$). So, the conclusion follows at once from Theorem \ref{thm:lfgu-left-duo-satisfies-ACCP}.
\end{proof}

\section*{Acknowledgments}
L.\,C.~was supported by the European Union's Horizon 2020 research and innovation programme under the Marie Sk\l{}odowska-Curie grant agreement No.~101021791. The same grant also financed S.\,T.'s visit at University of Graz in summer-fall 2022, during which the paper was written. The authors are both members of the National Group for Algebraic and Geometric Structures and their Applications (GNSAGA), a department of the Italian Mathematics Research Institute (INdAM). They are grateful to Victor Fadinger (University of Graz) for the construction in the last paragraph of Example \ref{exa:lfgu-premonoids}\ref{exa:lfgu-premonoids(3)}.

\end{document}